\documentclass[14pt]{amsart}
\usepackage{cases}
\usepackage{amsmath}
\usepackage{amsfonts}
\usepackage{bm}
\usepackage{amsfonts,amsmath,amssymb,amscd,bbm,amsthm,mathrsfs,dsfont}
\usepackage{mathrsfs}
\usepackage{pb-diagram}
\usepackage{color}
\usepackage{amssymb}
\usepackage{xypic}
\usepackage{mathrsfs}
\usepackage{pifont}
\usepackage[all]{xy}
\usepackage{breqn}
\usepackage{graphicx}
\usepackage{epstopdf}
\usepackage{float}

\newtheorem{Theorem}{Theorem}[section]
\newtheorem{Lemma}[Theorem]{Lemma}
\newtheorem{Definition}[Theorem]{Definition}
\newtheorem{Corollary}[Theorem]{Corollary}
\newtheorem{Proposition}[Theorem]{Proposition}

\newtheorem{Remark}[Theorem]{Remark}

\newtheorem{Observation}[Theorem]{Observation}

\newcommand{\lra}{\longrightarrow}

\newcommand{\ra}{\rightarrow}
\newcommand{\sdp}{\times\kern-.2em\vrule height1.1ex depth-.05ex}
\newcommand{\epi}{\lra \kern-.8em\ra}
\newcommand{\A}{{\mathcal A}}
\newcommand{\C}{{\mathbb C}}

\newcommand{\Z}{{\mathbb Z}}

\newcommand{\Rnum}[1]{\uppercase\expandafter{\romannumeral #1\relax}}

 \normalbaselines
\begin{document}

\title [Poisson structure and second quantization of quantum cluster algebras]
{Poisson structure and second quantization of \\quantum cluster algebras}
\author{Fang Li}\author{Jie Pan}

\address{Fang Li
\newline Department of Mathematics, Zhejiang University (Yuquan Campus), Hangzhou, Zhejiang 310027, P. R. China}
\email{fangli@zju.edu.cn}

\address{Jie Pan
\newline Department of Mathematics, Zhejiang University (Yuquan Campus), Hangzhou, Zhejiang 310027, P. R. China}
\email{panjie_zhejiang@qq.com}

\thanks{\textit{Mathematics Subject Classification(2010)}: 13F60, 46L65, 17B63}
\thanks{\textit{Keywords}: quantum cluster algebra, compatible Poisson structure, second quantization}
\date{version of \today}

\renewcommand{\thefootnote}{\alph{footnote}}
\maketitle
\bigskip
\begin{abstract}

Motivated by the phenomenon that compatible Poisson structures on a cluster algebra play a key role on its quantization (that is, quantum cluster algebra), we introduce the second quantization of a quantum cluster algebra, which means the correspondence between compatible Poisson structures of the quantum cluster algebra and its secondly quantized cluster algebras. Based on this observation, we find that a quantum cluster algebra possesses dual quantum cluster algebras such that their second quantization is essentially the same.

As an example, we give the secondly quantized cluster algebra $A_{p,q}(SL(2))$ of $Fun_{\C}(SL_{q}(2))$ in \S5.2.1 and show that it is a non-trivial second quantization, which may be realized as a parallel supplement to two parameters quantization of the general quantum group. Furthermore, we obtain a class of quantum cluster algebras with coefficients which possess a non-trivial second quantization. Its one special kind is quantum cluster algebras with almost principal coefficients with an additional condition.

Finally, we prove that the compatible Poisson structures of a quantum cluster algebra without coefficients is always a locally standard Poisson structure. Following this, it is shown that the second quantization of a quantum cluster algebra without coefficients is in fact trivial.

\end{abstract}
\vspace{4mm}

\tableofcontents

\section{Introduction and preliminaries}

The introduction of quantum cluster algebras \cite{BZ} is an important development of the theory of cluster algebras, which establishes a connection between cluster theory and the theory of quantum groups.

The theory of quantum groups comes from theoretic physics. It first appears in the inverse scattering method used to construct and solve quantum integrable systems and then is developed by Drinfield, Jimbo, etc. However, up to now, there is no unified definition for a quantum group.
Considering that the structures of quantum torus for each seed of a quantum cluster algebra are some kind of quantum groups, we can regard a quantum cluster algebra as a family of quantum groups with mutation actions.

It is often a complicated work to explore the quantization of a concrete Lie algebra or an associative algebra. Moreover, the complexity of some problems faced by mathematicians and physicists leads to the demand of two-parameters and multi-parameters quantum groups.

But the analogue of quantum groups with two or even multiple parameters had not been set up in the theory of cluster algebras until recently. It is still an open problem how to properly define quantum cluster algebras with two or even multiple parameters. For example, it is natural to consider to give the definition by simply adding one parameter to quantum cluster algebras in the form as follows:
\begin{equation}\label{I}
X_{i}X_{j}=p^{\delta_{ij}}q^{\lambda_{ij}}X^{e_{i}+e_{j}}.
\end{equation}
There has been interesting researches given in \cite{GY} and \cite{FH} focusing on this aspect. In \cite{GY}, the authors generalizes original quantum cluster algebras to such quantum cluster algebras with more than one parameters, and then proves that a very large class of quantum nilpotent algebras, defined axiomatically, admit this quantum cluster algebra structures. And in \cite{FH}, such quantum cluster algebras, called Toroidal cluster algebras, are found a profound relationship with quantum affine algebras under the sense of categorification.

However, on the other hand, as pointed out in \cite{GY} and \cite{FH}, the parameters $\delta_{ij}$ and $\lambda_{ij}$ in (\ref{I}) are independent, which means a two-parameters quantum cluster algebra so defined there is essentially determined by two parallel one-parameter quantum cluster algebras. In a sense, we may think such defined quantum cluster algebras with multiple parameters are `` {\em trivial} ". Hence, the new question is:

{\em How to give a more non-trivial definition of quantum cluster algebras with two or multiple parameters in the above sense}?

From our observation, the key lies in that how to meaningfully connect $\delta_{ij}$ with $\lambda_{ij}$ to make them not independent to each other. It is the main aim of this paper. We will focus on the correspondence between the quantization of a cluster algebra and the compatible Poisson structure on it(\cite{BZ},\cite{N}), and lift this fact to non-commutative level so as to define a kind of quantum algebras with two parameters, which are not trivial in the above sense. We call such algebras as {\em secondly quantized cluster algebras}. We can continue this way to discuss the possibility of higher quantization, which will be mentioned in the sequel.\vspace{2mm}

Poisson geometry and its related Poisson algebra came from mechanics in nineteenth century named after the famous mathematical physicist S.D.Poisson. Due to its connections with classical mechanics, symplectic geometry, basing on integrable system and the correspondence between the quantization of a cluster algebra and the compatible Poisson structure on the cluster algebra, we will be able to find the background of (secondly) quantized cluster algebras in these research areas.

For $n\leqslant m\in\mathbb{N}$, denote $\mathbb{T}_{n}$ the \textbf{$n$-regular tree} with vertices $t\in \mathbb{T}_{n}$.
Let $\mathscr{F}$ be the field of rational functions over $\mathbb{Q}$ in $m$ independent variables.
\begin{Definition}
(1)\; A \textbf{seed} at vertex $t\in\mathbb{T}_n$ is a pair $\Sigma=(\tilde{\mathcal{X}}(t),\tilde{B}(t))$ such that\\
$\bullet$\; $\tilde{\mathcal{X}}(t)=(x_{1;t},x_{2;t},\cdots,x_{m;t})$ is an $m$-tuple satisfying that the elements form a free generating set of $\mathscr{F}$;\\
$\bullet$\; $\tilde{B}(t)$ is an $m\times n$ integer matrix such that the principal part is skew-symmetrizable, i.e. there is a positive diagonal matrix $D$ satisfying $DB(t)$ is skew-symmetric, where $B(t)$ is the first $n$ rows of $\tilde{B}(t)$.

(2)\; For any $k\in[1,n]$, define the \textbf{mutation} $\mu_{k}$ at direction $k$ satisfying that $\mu_{k}(\Sigma)=\Sigma^{\prime}=(\tilde{\mathcal{X}}^{\prime},\tilde{B}^{\prime})$,
\begin{equation*}
\mu_{k}(x_{k;t})=\frac{\prod\limits_{i=1}^{m}x_{i;t}^{[b_{ik}^{t}]_{+}}+\prod\limits_{i=1}^{m}x_{i;t}^{[-b_{ik}^{t}]_{+}}}{x_{k;t}}
\end{equation*}
where $[a]_{+}=max\left\{a,0\right\}$ for $a\in\mathbb{R}$. And
\[\tilde{\mathcal{X}}^{\prime}=(\tilde{\mathcal{X}}(t))\backslash\left\{x_{k;t}\right\})\bigcup\left\{\mu_{k}(x_{k;t})\right\}.\]
\[\tilde{B}^{\prime}=\mu_k(\tilde B(t))=(b_{ij}^{\prime})_{m\times n}\] satisfying that
\begin{equation}\label{matrixmut}
b_{ij}^{\prime}=\left\{
\begin{array}{lcr}
-b_{ij}(t)&&if\quad i=k \quad or\quad j=k\\
b_{ij}(t)+sgn(b_{ik}(t))[b_{ik}(t)b_{kj}(t)]_{+}&&otherwise
\end{array}
\right .
\end{equation}
\end{Definition}
It can be proved that $\mu_{k}$ is an involution.
\begin{Definition}
Given seeds $\Sigma(t)=(\tilde{\mathcal{X}}(t),\tilde{B}(t))$ at $t\in \mathbb{T}_{n}$ so that $\Sigma(t^{\prime})=\mu_{k}(\Sigma(t))$ for any $t-t^{\prime}$ in $\mathbb{T}_{n}$ connected by an edge labeled $k\in[1,n]$, then the $\mathbb{Z}[x_{n+1}^{\pm1},\cdots,x_{m}^{\pm 1}]$-subalgebra of $\mathscr{F}$ generated by all variables in $\bigcup\limits_{t\in\mathbb{T}_{n}}\mathcal{X}(t)$ is called the \textbf{cluster algebra} $A(\Sigma)$ (or simply $A$) associated with $\Sigma$.
\end{Definition}

Besides, we also introduce the concept of quantum cluster algebras. For a vertex $t_{0}\in\mathbb{T}_n$, let $\Lambda(t_{0})=(\lambda_{ij})_{m\times m}$ be a skew-symmetric integer matrix satisfying
\begin{equation}\label{2}
\tilde{B}(t_{0})^{\top}\Lambda(t_{0})=
\begin{pmatrix}
D& O
\end{pmatrix}_{n\times m}
\tag{C1}
\end{equation}
Then $(\tilde{B}(t_{0}),\Lambda(t_{0}))$ is called a \textbf{compatible pair}. Let $\left\{e_{i}\right\}_{i=1}^{m}$ be the standard basis for $\mathbb{Z}^{m}$. Define a skew-symmetric bilinear form $\Lambda_{t_{0}}:\mathbb{Z}^{m}\times\mathbb{Z}^{m}\rightarrow\mathbb{Z}$ satisfying that
\[\Lambda_{t_{0}}(e,f)=\sum \limits_{i,j=1}^{m}a_{i}b_{j}\Lambda_{t_{0}}(e_{i},e_{j})=\sum \limits_{i,j=1}^{m}a_{i}b_{j}\lambda_{ij},\]
where $e=\sum \limits_{i=1}^{m}a_{i}e_{i},f=\sum \limits_{j=1}^{m}b_{j}e_{j}$.

Give a set of variables
\[\tilde X(t_{0})=\left\{X_{t_{0}}^{e_{1}},\cdots,X_{t_{0}}^{e_{n}},X^{e_{n+1}},\cdots,X^{e_{m}}\right\}\]
called the \textbf{(extended) cluster} at $t_{0}$, where $X_{t_{0}}^{e_{i}},i\in[1,n]$ are called the \textbf{cluster variables} at $t_{0}$ while $X^{e_{i}},i\in[n+1,m]$ are called \textbf{frozen variables}.

For the Laurent polynomial ring $\mathbb{Z}[q^{\pm\frac{1}{2}}]$ with a formal variable $q$, define a $\mathbb{Z}[q^{\pm\frac{1}{2}}]$-algebra $T_{t_{0}}$ generated by $X(t_{0})$ satisfying the following relations:
\[X_{t_{0}}^{e_{i}}X_{t_{0}}^{e_{j}}=q^{\frac{1}{2}\lambda_{ij}}X_{t_{0}}^{e_{i}+e_{j}},\forall i,j\in[1,m]\]
We call $T_{t_{0}}$ the \textbf{quantum torus} at $t_{0}$. Denoted by $\mathscr{F}_{q}$ the skew-field of fractions of $T_{t_{0}}$.

In general, for any $e\in\mathbb{Z}^{m}$, let $X_{t_{0}}^{e}$ denote the variable corresponding to $e$. Due to the bilinearity of $\Lambda_{t_{0}}$ and the fact that $e$ is generated by $\left\{e_{i}\mid i\in[1,m]\right\}$, we obtain that
\begin{equation}\label{1}
X_{t_{0}}^{e}X_{t_{0}}^{f}=q^{\frac{1}{2}\Lambda_{t_{0}}(e,f)}X_{t_{0}}^{e+f}
\end{equation}

\begin{Definition}
[\cite{BZ}]
(i)\; Given a fixed $t_{0}\in \mathbb{T}_{n}$, we denote $\Sigma(t_{0})=(\tilde{X}(t_{0}),\tilde{B}(t_{0}),\Lambda(t_{0}))$ an \textbf{initial quantum seed}.

(ii)\; Let $t\in\mathbb{T}_{n}$ be an adjacent vertex of $t_{0}$, i.e. $t-t_{0}$ is an edge in $\mathbb{T}_{n}$ labeled $k\in[1,n]$. Let $b_{k}(t_{0})$ be the $k$-th column of $\tilde{B}(t_{0})$. Define the \textbf{mutation} $\mu_{k}$ at direction $k$ satisfying that
\begin{equation*}
X_{t}^{e_{k}}=\mu_{k}(X_{t_{0}}^{e_{k}})=X_{t_{0}}^{-e_{k}+[b_{k}(t_{0})]_{+}}+X_{t_{0}}^{-e_{k}+[-b_{k}(t_{0})]_{+}}
\end{equation*}
such that
\[\tilde{X}(t)=(\tilde{X}(t_{0})\backslash\left\{X^{e_{k}}(t_{0})\right\})\bigcup\left\{X_{t}^{e_{k}}\right\}.\]
where $\tilde B(t)=\mu_k(\tilde B(t_0))$ is the same as that according to (\ref{matrixmut}).
And, $\Lambda(t)=\mu_k(\Lambda(t_0))=(\lambda_{ij}(t))_{m\times m}$ where
\begin{equation}\label{lambdamutation}
\lambda_{ij}(t)=\left\{
\begin{array}{lcr}
-\lambda_{kj}(t_{0})+\sum\limits_{l=1}^{m}[b_{lk}(t_{0})]_{+}\lambda_{lj}(t_{0})&&if\quad i=k\neq j\\
-\lambda_{ik}(t_{0})+\sum\limits_{l=1}^{m}[b_{lk}(t_{0})]_{+}\lambda_{il}(t_{0})&&if\quad j=k\neq i\\
\lambda_{ij}(t_{0})&&otherwise
\end{array}
\right .
\end{equation}

\end{Definition}
It can be proved that the seed $\Sigma(t)=(\tilde{X}(t),\tilde{B}(t),\Lambda(t))$ at $t$ still satisfies the relation (\ref{1}) and (\ref{2}) and $\mu_{k}$ is an involution.

It can be seen that in the quantum case, because of the relation (\ref{2}), $\tilde{B}(t)$ is always of full column rank $n$.
\begin{Definition}
$\left(\right.$\cite{BZ}$\left.\right)$ Given seeds $\Sigma(t)=(\tilde{X}(t),\tilde{B}(t),\Lambda(t))$ at $t\in \mathbb{T}_{n}$, if $\Sigma(t)$ and $\Sigma(t^{\prime})$ can do mutation to each other for any adjacent pair of vertices $t-t^{\prime}$ in $\mathbb{T}_{n}$, then the $\mathbb{Z}[q^{\pm\frac{1}{2}}][X^{\pm e_{n+1}},\cdots,X^{\pm e_{m}}]$-subalgebra of $\mathscr{F}_{q}$ generated by all variables in $\bigcup\limits_{t\in\mathbb{T}_{n}}X(t)$ is called the \textbf{quantum cluster algebra} $A_{q}(\Sigma)$ (or simply $A_{q}$) associated with $\Sigma$.
\end{Definition}

We call $\Lambda(t)$ the {\bf deformation matrix} of this quantum cluster algebra $A_q$ at $t\in \mathbb T_n$.

Now we explain the relationship between cluster algebras in non-quantum and quantum cases from their definitions.

For a quantum cluster algebra $A_q$, let $q\rightarrow 1$. Then we obtain a \emph{non-quantum} cluster algebra $A$, which is called {\bf the correspondent classical version} (briefly, {\bf CCV}) of $A_q$; conversely, $A_q$ is called {\bf the correspondent quantum version} (briefly, {\bf CQV}) of $A$. Under this relationship, the (extended) clusters $\tilde X(t)=\{X_{t}^{e_{1}}, \cdots, X_{t}^{e_{m}}\}$ in $A_q$ and the (extended) clusters $\tilde {\mathcal X}(t)=\{x_{1;t}, \cdots, x_{m;t}\}$ in $A$ correspond to each other.

Note that since the rank of $\tilde{B}(t)$ of $A_q$ is always $n$ as mentioned above, only those cluster algebras $A$ whose $\tilde{B}(t)$ are of rank $n$ have the correspondent quantum versions.

Due to our motivation for this work, we will first discuss Poisson structures on quantum cluster algebras. So here, we recall the concepts and notations of Poisson structures.

A \textbf{Poisson structure} on an associative k-algebra $\A$ is a triple $(\A,\cdot,\left\{-,-\right\})$ where $(\A,\left\{-,-\right\})$ is a Lie k-algebra i.e. satisfying Jacobi identity such that the Leibniz rule holds: for any $a,b,c\in \A$,
\[\left\{a,bc\right\}=\left\{a,b\right\}c+b\left\{a,c\right\}.\]
Algebra $\A$ together with a Poisson structure on it is called a \textbf{Poisson algebra}. Denote the \textbf{Hamiltonian} of $a\in \A$ by
\[ham(a)=\left\{a,-\right\}\in End_{k}(\A,\A).\]
Then the Leibniz rule is equivalent to that $ham(a)$ is a derivation of $\A$ as an associative algebra for any $a\in \A$.
\begin{Definition}
Let $\A$ be an associative algebra. $[a,b]=ab-ba$ is called the \textbf{commutator} of a and b, for any $a,b\in \A$. And for any $\lambda \in k$, $(\A,\cdot,\lambda[-,-])$ is a Poisson algebra called a \textbf{standard Poisson structure} on $(\A,\cdot)$.
\end{Definition}

As we know, so far only the Poisson structure of (non-quantum) cluster algebras has been studied, e.g. see \cite{GSV} and \cite{N}. We recall the following notions from \cite{GSV}:

(1)\; For a cluster algebra $A$, one of its extended cluster $\tilde{\mathcal X}=(x_{1},\cdots,x_{m})$ is said to be \textbf{log-canonical} with respect to a Poisson structure $(A,\cdot,\left\{-,-\right\})$ if $\left\{x_{i},x_{j} \right\}=\psi_{ij}x_{i}x_{j}$, where $\psi_{ij}\in \mathbb{Z}$ for any $i,j\in[1,m]$.

(2)\; A Poisson structure $\left\{-,-\right\}$ on a cluster algebra $A$ is called \textbf{compatible} with $A$ if all clusters in $A$ are log-canonical with respect to $\left\{ -,-\right\}$. In this case, the matrix $\Psi=(\psi_{ij})_{m\times m}$ is called the {\bf Poisson matrix} associated to cluster $\tilde{\mathcal X}$ (with respect to the Poisson structure).

From \cite{N}, we know that a compatible Poisson structure is given on a cluster algebra $A$ via a family of $\Psi(t)$ ($t\in \mathbb{T}_n$) as Poisson matrices such that following mutation formula of $\Psi(t)$ holds for each adjacent vertex pair $(t, t')$ in $\mathbb{T}_n$ connected by an edge labeled $k$:

\begin{equation}\label{non-quantum omega mutation}
\psi_{ij}(t^{\prime})=\left\{
\begin{array}{lcr}
-\psi_{kj}(t)+\sum\limits_{l=1}^{m}[b_{lk}(t)]_{+}\psi_{lj}(t)&&\text{if}\; i=k\neq j\\
-\psi_{ik}(t)+\sum\limits_{l=1}^{m}[b_{lk}(t)]_{+}\psi_{il}(t)&&\text{if}\; j=k\neq i\\
\psi_{ij}(t)&& \text{otherwise}
\end{array}
\right .
\end{equation}
And for a compatible Poisson structure, there is always $\tilde{B}^{\top}(t)\Psi(t)=(D\text{ }0)$ for any $t\in\mathbb{T}_{n}$.

We similarly define the compatibility of non-commutative Poisson structures on a quantum cluster algebra.

\begin{Definition}
(1)\; For a quantum cluster algebra $A_{q}$, one of its extended cluster $\tilde X(t)=(X_{1},\cdots,X_{m})$ at $t\in \mathbb{T}_n$ is said to be \textbf{log-canonical} with respect to a Poisson structure $(A_{q},\cdot,\left\{-,-\right\})$ if $\left\{X_{i},X_{j} \right\}=\omega_{ij}X^{e_{i}+e_{j}}$, where $\omega_{ij}\in \mathbb{Z}[q^{\pm \frac{1}{2}}]$ for any $i,j\in[1,m]$. In this case, the matrix $\Omega(t)=(\omega_{ij})_{m\times m}$ is called the {\bf Poisson matrix} associated to the cluster $\tilde X(t)$ (with respect to the Poisson structure).

(2)\; A Poisson structure $\left\{-,-\right\}$ on a quantum cluster algebra $A_{q}$ is called \textbf{compatible} with $A_{q}$ if all clusters in $A_{q}$ are log-canonical with respect to $\left\{ -,-\right\}$.
\end{Definition}

Trivially, Poisson matrices are always skew-symmetric for either commutative or non-commutative cases.

The paper is organized as follows.

In Section 2 we give the mutation formula of Poisson matrices in a quantum cluster algebra $A_q$ (Theorem \ref{pmm}) and the equivalent characterization for two adjacent clusters to be log-canonical with a Poisson structure on $A_q$ (Theorem \ref{exchange}).

In Section 3, we introduce in Definition \ref{IIq} the concept of the second quantization of a quantum cluster algebra $A_q$ based on the correspondence between Poisson matrices and (second) deformation matrices, which means the correspondence between compatible Poisson structures and (secondly) quantized cluster algebras. And Proposition \ref{BW=D} is given here to show that a second deformation matrix must satisfy (\ref{add}) which is the same as (\ref{2}). Following this proposition, we find in Theorem \ref{twoways} that a quantum cluster algebra possesses dual quantum cluster algebras such that their second quantization are essentially the same.

In Section 4, we prove Lemma \ref{zero}, in which the condition (C4) for all seeds is transformed to the condition $\hat{\Lambda}\tilde{B}=c\tilde{D}$ for an arbitrary seed.

In Section 5, we introduce the cluster decomposition of a quantum cluster algebra. Then with the help of Lemma \ref{zero}, the compatibility of a Poisson structure comes down to the cluster indecomposable case, see Proposition \ref{decoposition to pieces}. And so does second quantization (Theorem \ref{cluster decomposition}).

As an example, we give the secondly quantized cluster structure $A_{p,q}(SL(2))$ of $Fun_{\C}(SL_{q}(2))$ in \S5.2.1 and show that it is a non-trivial second quantization, which may be realized as a parallel supplement to two parameters quantization of the general quantum group. After that, via a cluster extension, we obtain a class of quantum cluster algebras which possess a non-trivial second quantization (Theorem \ref{ext}). One special kind of it is quantum cluster algebras with almost principal coefficients with an additional condition (Corollary \ref{almost}).

In Section 6, it is shown that a Poisson structure is compatible with a quantum cluster algebra without coefficients if and only if it is a locally standard Poisson structure (Theorem \ref{r2}). Therefore the second quantization of a quantum cluster algebra without coefficients is always trivial (Corollary \ref{trivial})

\vspace{4mm}

\section{Compatible Poisson structures on quantum cluster algebras and mutation of Poisson matrices}

In \cite{GSV}, compatible Poisson structures on cluster algebras are characterized and moreover, such structures are constructed on Grassmannians. In this section and Section 5, we will discuss Poisson structures compatible with a quantum cluster algebra.

Here and in the following, we always assume that in a quantum cluster algebra $A_q$, the initial quantum seed at $t_{0}$ is $(\tilde{X},\tilde{B},\Lambda)$, where $\tilde{X}=(X_{1},X_{2},\ldots,X_{m})$ (denote $X_{i}=X^{e_{i}},i\in[1,m]$) with the first $n$ variables mutable, $\tilde{B}$ is an $m \times n$ skew-symmetrizable integer matrix with skew-symmetizer $D$ and $\Lambda$ is an $m \times m$ skew-symmetric matrix such that $(\tilde{B},\Lambda)$ is a compatible pair.

First of all, notice that if $\left\{-,-\right\}$ is trivial, i.e.$\left\{X,Y\right\}=0$ for any $X,Y\in A_{q}$, then $\omega_{ij}$ are all 0, thus it is naturally compatible with $A_{q}$. Therefore in the following we only consider about nontrivial Poisson structures.
\begin{Theorem}\label{pmm}
For a quantum cluster algebra $A_q$, if its quantum seed $\tilde{X}$ and $\mu_{k}(\tilde{X})$ are log-canonical with a nontrivial Poisson structure $\left\{-,-\right\}$ and the Poisson matrices associated to them are $\Omega=(\omega_{ij})_{m\times m}$ and $\Omega^{\prime}=(\omega_{ij}^{\prime})_{m\times m}$ respectively, then

(1)\; for any $j\neq k$, where $j\in[1,m]$ while $k\in [1,n]$, we have
\begin{equation}\label{**}
\begin{array}{l}
\quad \sum \limits_{b_{tk}>0}(\omega_{tj}q^{\frac{1}{2}\lambda_{jt}}\sum \limits_{h=1}^{[b_{tk}]_{+}}q^{\sum \limits_{i=t}^{m}([b_{ik}]_{+}-\delta_{ik})\lambda_{ji}-h\lambda_{jt}})-\omega_{kj}q^{\frac{1}{2}\lambda_{kj}+\sum \limits_{i=k+1}^{m}\lambda_{ji}[b_{ik}]_{+}}\\
=\sum \limits_{b_{tk}<0}(\omega_{tj}q^{\frac{1}{2}\lambda_{jt}}\sum \limits_{h=1}^{[-b_{tk}]_{+}}q^{\sum \limits_{i=t}^{m}([-b_{ik}]_{+}-\delta_{ik})\lambda_{ji}-h\lambda_{jt}})-\omega_{kj}q^{\frac{1}{2}\lambda_{kj}+\sum \limits_{i=k+1}^{m}\lambda_{ji}[-b_{ik}]_{+}}
\end{array}
\end{equation}
where $\delta_{ij}$ equals 1 when $i=j$ and 0 otherwise.

(2)\; the mutation formula of Poisson matrices $\Omega$ in direction $k$ is given as follows:
\begin{equation}\label{omega mutation formula}
\omega_{ij}^{\prime}=
\left\{
\begin{array}{lcr}
q^{\frac{1}{2}(\lambda_{jk}-\sum \limits_{t=1}^{m}[b_{tk}]_{+}\lambda_{jt})}H &&if\quad i=k\neq j\\
-\omega_{ki}^{\prime}&&if\quad j=k\neq i\\
\omega_{ij}&&otherwise
\end{array}
\right .
\end{equation}
where $H$ denotes the left or right side of (\ref{**}).

\end{Theorem}

\begin{proof}
(1)\; For any $k\in[1,n]$, let $\tilde{X}^{\prime}=\mu_{k}(\tilde{X})=(X_{1},\cdots,X_{k-1},X_{k}^{\prime},\cdots,X_{m})$. By the assumption, $\tilde{X}^{\prime}$ is log-canonical with respect to $\left\{-,-\right\}$. Therefore $\left\{X_{k}^{\prime},X_{j} \right\}=\omega_{kj}^{\prime} {X^{\prime}}^{e_{k}+e_{j}}$ for some $\omega_{kj}^{\prime}\in \mathbb{Z}[q^{\pm \frac{1}{2}}]$ for any $j\neq k\in[1,m]$.

By the exchange relations for quantum cluster algebras, we obtain that

$ \left\{X_{k}^{\prime},X_{j} \right\}
=\left\{X^{-e_{k}+\sum \limits_{b_{ik}>0}b_{ik}e_{i}}, X_{j} \right\} + \left\{X^{-e_{k}-\sum \limits_{b_{ik}<0}b_{ik}e_{i}}, X_{j} \right\} $
\begin{equation}\label{*}
=q^{\frac{1}{2}\sum \limits_{l>h}\lambda_{lh}([b_{lk}]_{+}-\delta_{lk})([b_{hk}]_{+}-\delta_{hk})}\left\{\prod \limits_{i=1}^{m}X_{i}^{[b_{ik}]_{+}-\delta_{ik}}, X_{j} \right\}
+q^{\frac{1}{2}\sum \limits_{l>h}\lambda_{lh}([-b_{lk}]_{+}-\delta_{lk})([-b_{hk}]_{+}-\delta_{hk})}\left\{\prod \limits_{i=1}^{m}X_{i}^{[-b_{ik}]_{+}-\delta_{ik}},X_{j}\right\},
\end{equation}
and

$ \left\{\prod \limits_{i=1}^{m}X_{i}^{[b_{ik}]_{+}-\delta_{ik}},X_{j}\right\} $\\
$ =\sum \limits_{b_{tk}>0}\prod\limits_{i=1}^{t-1}X_{i}^{[b_{ik}]_{+}-\delta_{ik}}\left\{X_{t}^{[b_{tk}]_{+}},X_{j}\right\}\prod \limits_{i=t+1}^{m}X_{i}^{[b_{ik}]_{+}-\delta_{ik}}-\prod \limits_{i=1}^{k-1}X_{i}^{[b_{ik}]_{+}}\left\{X_{k}^{-1},X_{j}\right\}\prod \limits_{i=k+1}^{m}X_{i}^{[b_{ik}]_{+}} $ \\
$=\sum \limits_{b_{tk}>0}\sum\limits_{h=0}^{[b_{tk}]_{+}-1}\prod \limits_{i=1}^{t-1}X_{i}^{[b_{ik}]_{+}-\delta_{ik}}X_{t}^{h}\left\{X_{t},X_{j}\right\}X_{t}^{[b_{tk}]_{+}-h-1}\prod \limits_{i=t+1}^{m}X_{i}^{[b_{ik}]_{+}-\delta_{ik}} $\\
$ .\;\;\;\;\;\;\;\;\;\;\;\;\;\;\;\;\;\;\;\;\;\; -\prod \limits_{i=1}^{k-1}X_{i}^{[b_{ik}]_{+}}X_{k}^{-1}\left\{X_{k},X_{j}\right\}X_{k}^{-1}\prod \limits_{i=k+1}^{m}X_{i}^{[b_{ik}]_{+}} \;\;\;\;\;\;(\text{say}\;\prod \limits_{i=1}^{0}M = 1,\forall M) $\\
$ =\sum \limits_{b_{tk}>0}\sum\limits_{h=0}^{[b_{tk}]_{+}-1}\omega_{tj}\prod \limits_{i=1}^{t-1}X_{i}^{[b_{ik}]_{+}-\delta_{ik}}X_{t}^{h}X^{e_{t}+e_{j}}X_{t}^{[b_{tk}]_{+}-h-1}\prod \limits_{i=t+1}^{m}X_{i}^{[b_{ik}]_{+}-\delta_{ik}}
-\omega_{kj}\prod \limits_{i=1}^{k-1}X_{i}^{[b_{ik}]_{+}}X_{k}^{-1}X^{e_{k}+e_{j}}X_{k}^{-1}\prod \limits_{i=k+1}^{m}X_{i}^{[b_{ik}]_{+}} $\\
$ =\sum \limits_{b_{tk}>0}\sum\limits_{h=0}^{[b_{tk}]_{+}-1}\omega_{tj}q^{\frac{1}{2}\lambda_{jt}}\prod \limits_{i=1}^{t-1}X_{i}^{[b_{ik}]_{+}-\delta_{ik}}X_{t}^{h+1}X_{j}X_{t}^{[b_{tk}]_{+}-h-1}\prod \limits_{i=t+1}^{m}X_{i}^{[b_{ik}]_{+}-\delta_{ik}}
-\omega_{kj}q^{\frac{1}{2}\lambda_{kj}}\prod \limits_{i=1}^{k-1}X_{i}^{[b_{ik}]_{+}}X_{k}^{-1}X_{j}\prod \limits_{i=k+1}^{m}X_{i}^{[b_{ik}]_{+}} $\\
$ =(\sum \limits_{b_{tk}>0}\sum\limits_{h=0}^{[b_{tk}]_{+}-1}\omega_{tj}q^{\frac{1}{2}\lambda_{jt}+([b_{tk}]_{+}-h-1)\lambda_{jt}+\sum \limits_{i=t+1}^{m}\lambda_{ji}([b_{ik}]_{+}-\delta_{ik})}-\omega_{kj}q^{\frac{1}{2}\lambda_{kj}+\sum \limits_{i=k+1}^{m}\lambda_{ji}[b_{ik}]_{+}})\prod \limits_{i=1}^{m}X_{i}^{[b_{ik}]_{+}-\delta_{ik}}X_{j} $\\
$ =(\sum \limits_{b_{tk}>0}\sum\limits_{h=1}^{[b_{tk}]_{+}}\omega_{tj}q^{\frac{1}{2}\lambda_{jt}+\sum
\limits_{i=t}^{m}\lambda_{ji}([b_{ik}]_{+}-\delta_{ik})-h\lambda_{jt}}-\omega_{kj}q^{\frac{1}{2}\lambda_{kj}+\sum \limits_{i=k+1}^{m}\lambda_{ji}[b_{ik}]_{+}})\prod \limits_{i=1}^{m}X_{i}^{[b_{ik}]_{+}-\delta_{ik}}X_{j}. $
\vspace{3mm}\\
And similarly,
\begin{equation*}
\left\{\prod \limits_{i=1}^{m}X_{i}^{[-b_{ik}]_{+}-\delta_{ik}},X_{j}\right\}=(\sum \limits_{b_{tk}<0}\sum\limits_{h=1}^{[-b_{tk}]_{+}}\omega_{tj}q^{\frac{1}{2}\lambda_{jt}+\sum
\limits_{i=t}^{m}\lambda_{ji}([-b_{ik}]_{+}-\delta_{ik})-h\lambda_{jt}}-\omega_{kj}q^{\frac{1}{2}\lambda_{kj}+\sum \limits_{i=k+1}^{m}\lambda_{ji}[-b_{ik}]_{+}})\prod \limits_{i=1}^{m}X_{i}^{[-b_{ik}]_{+}-\delta_{ik}}X_{j}
\end{equation*}
Therefore,
\begin{equation*}
\begin{array}{ll}
\left\{X_{k}^{\prime},X_{j}\right\}=&q^{\frac{1}{2}\sum \limits_{l>h}\lambda_{lh}([b_{lk}]_{+}-\delta_{lk})([b_{hk}]_{+}-\delta_{hk})}(\sum \limits_{b_{tk}>0}\sum\limits_{h=1}^{[b_{tk}]_{+}}\omega_{tj}q^{\frac{1}{2}\lambda_{jt}+ \sum\limits_{i=t}^{m}\lambda_{ji}([b_{ik}]_{+}-\delta_{ik})-h\lambda_{jt}}\\
& \quad\quad\quad\quad\quad\quad\quad\quad\quad\quad\quad\quad\quad\quad\quad\quad\quad-\omega_{kj}q^{\frac{1}{2}\lambda_{kj}+\sum \limits_{i=k+1}^{m}\lambda_{ji}[b_{ik}]_{+}})\prod \limits_{i=1}^{m}X_{i}^{[b_{ik}]_{+}-\delta_{ik}}X_{j}\\
&+q^{\frac{1}{2}\sum \limits_{l>h}\lambda_{lh}([-b_{lk}]_{+}-\delta_{lk})([-b_{hk}]_{+}-\delta_{hk})}(\sum \limits_{b_{tk}<0}\sum\limits_{h=1}^{[-b_{tk}]_{+}}\omega_{tj}q^{\frac{1}{2}\lambda_{jt}+ \sum\limits_{i=t}^{m}\lambda_{ji}([-b_{ik}]_{+}-\delta_{ik})-h\lambda_{jt}} \\
& \quad\quad\quad\quad\quad\quad\quad\quad\quad\quad\quad\quad\quad\quad\quad\quad-\omega_{kj}q^{\frac{1}{2}\lambda_{kj}+\sum \limits_{i=k+1}^{m}\lambda_{ji}[-b_{ik}]_{+}})\prod \limits_{i=1}^{m}X_{i}^{[-b_{ik}]_{+}-\delta_{ik}}X_{j}
\end{array}
\end{equation*}
$ =(\sum \limits_{b_{tk}>0}\sum\limits_{h=1}^{[b_{tk}]_{+}}\omega_{tj}q^{\frac{1}{2}\lambda_{jt}+\sum
\limits_{i=t}^{m}\lambda_{ji}([b_{ik}]_{+}-\delta_{ik})-h\lambda_{jt}}-\omega_{kj}q^{\frac{1}{2}\lambda_{kj}+\sum \limits_{i=k+1}^{m}\lambda_{ji}[b_{ik}]_{+}})X^{-e_{k}+\sum \limits_{b_{ik}>0}b_{ik}e_{i}}X_{j}$
$$
+(\sum \limits_{b_{tk}<0}\sum\limits_{h=1}^{[-b_{tk}]_{+}}\omega_{tj}q^{\frac{1}{2}\lambda_{jt}+\sum
\limits_{i=t}^{m}\lambda_{ji}([-b_{ik}]_{+}-\delta_{ik})-h\lambda_{jt}}-\omega_{kj}q^{\frac{1}{2}\lambda_{kj}+\sum \limits_{i=k+1}^{m}\lambda_{ji}[-b_{ik}]_{+}})X^{-e_{k}-\sum \limits_{b_{ik}<0}b_{ik}e_{i}}X_{j}.\;\;\;\;(\spadesuit)$$

On the other hand,
\begin{equation}\label{omega}
\omega_{kj}^{\prime}{X^{\prime}}^{e_{k}+e_{j}}
=\omega_{kj}^{\prime}q^{\frac{1}{2}\lambda_{jk}^{\prime}}X_{k}^{\prime}X_{j}
=\omega_{kj}^{\prime}q^{\frac{1}{2}(\sum \limits_{t=1}^{m}[b_{tk}]_{+}\lambda_{jt}-\lambda_{jk})}(X^{-e_{k}+\sum \limits_{b_{ik}>0}b_{ik}e_{i}}+X^{-e_{k}-\sum \limits_{b_{ik}<0}b_{ik}e_{i}})X_{j}.
\end{equation}

Since $\{X_{k}^{\prime},X_{j}\}=\omega_{kj}^{\prime}{X^{\prime}}^{e_{k}+e_{j}}$ and the fact that
cluster Laurent monomials in a cluster $X$ are $\mathbb{Z}[q^{\pm \frac{1}{2}}]$-linear independent, we compare the coefficients of the corresponding cluster Laurent monomials in ($\spadesuit$) and the right-side of (\ref{omega}), it follows that

$\sum \limits_{b_{tk}>0}\sum\limits_{h=1}^{[b_{tk}]_{+}}\omega_{tj}q^{\frac{1}{2}\lambda_{jt}+\sum
\limits_{i=t}^{m}\lambda_{ji}([b_{ik}]_{+}-\delta_{ik})-h\lambda_{jt}}-\omega_{kj}q^{\frac{1}{2}\lambda_{kj}+\sum \limits_{i=k+1}^{m}\lambda_{ji}[b_{ik}]_{+}}= \omega_{kj}^{\prime}q^{\frac{1}{2}(\sum \limits_{t=1}^{m}[b_{tk}]_{+}\lambda_{jt}-\lambda_{jk})},$ \\
also, $\sum \limits_{b_{tk}<0}\sum\limits_{h=1}^{[-b_{tk}]_{+}}\omega_{tj}q^{\frac{1}{2}\lambda_{jt}+\sum
\limits_{i=t}^{m}\lambda_{ji}([-b_{ik}]_{+}-\delta_{ik})-h\lambda_{jt}}-\omega_{kj}q^{\frac{1}{2}\lambda_{kj}+\sum \limits_{i=k+1}^{m}\lambda_{ji}[-b_{ik}]_{+}}= \omega_{kj}^{\prime}q^{\frac{1}{2}(\sum \limits_{t=1}^{m}[b_{tk}]_{+}\lambda_{jt}-\lambda_{jk})}.$

So (\ref{**}) is satisfied.

(2)\; It can also be seen from above equations that when (\ref{**}) is satisfied, $\omega_{kj}^{\prime}=q^{\frac{1}{2}(\lambda_{jk}-\sum \limits_{t=1}^{m}[b_{tk}]_{+}\lambda_{jt})}H$. Similar for $\omega_{ik}^{\prime}$. And as $X_{i},X_{j}$ do not change in mutation at direction $k$ when $i,j\neq k$, $\omega_{ij}^{\prime}=\omega_{ij}$.
\end{proof}

\begin{Remark}
A cluster algebra can be regarded as a quantum cluster algebra with $q=1$ or $\Lambda=0$. Then we can see in this case that above mutation formula of Poisson matrices coincides with the mutation formula of Poisson matrices for a cluster algebra in (\ref{non-quantum omega mutation}).
\end{Remark}

{\bf In the sequel, we will always use $\frac{r}{s}=\frac{u}{v}$ to represent $rv=su$ no matter $s$ or $v$ equals to $0$ or not.}
\begin{Lemma}\label{b}
Under the same condition as that of Theorem \ref{pmm}, let $u, v, j\in [1,m]$, $k\in[1,n]$ and $j\neq k$.

(i)\;If $b_{uk}\neq0$, then $\frac{\omega_{uj}}{\omega_{kj}}=\frac{q^{\frac{1}{2}\lambda_{u j}}-q^{\frac{1}{2}\lambda_{ju}}}{q^{\frac{1}{2}\lambda_{kj}}-q^{\frac{1}{2}\lambda_{jk}}}$.

(ii)\;If $b_{uk}b_{vk}\neq0$, then $\frac{\omega_{uj}}{\omega_{vj}}=\frac{q^{\frac{1}{2}\lambda_{uj}}-q^{\frac{1}{2}\lambda_{ju}}}{q^{\frac{1}{2}\lambda_{vj}}- q^{\frac{1}{2}\lambda_{jv}}}$.
\end{Lemma}
\begin{proof}
In the proof of Theorem \ref{pmm}, if we choose $\left\{p_{1},\cdots,p_{m}\right\}$ another permutation of $[1,m]$ instead of $[1,m]$ in (\ref{*}), we will finally obtain an equation different from (\ref{**}) as
\begin{equation*}
\left\{X^{-e_{k}+\sum \limits_{b_{ik}>0}b_{ik}e_{i}},X_{j} \right\}
=q^{\frac{1}{2}\sum \limits_{l>h}\lambda_{p_{l}p_{h}}}([b_{p_{l}k}]_{+}-\delta_{p_{l}k})([b_{p_{h}k}]_{+}-\delta_{p_{h}k})\left\{\prod \limits_{i=1}^{m}X_{p_{i}}^{[b_{p_{i}k}]_{+}-\delta_{p_{i}k}},X_{j}\right\}
\end{equation*}
Denote $p^{-1}(k)=i\in [1,m]$ when $p_{i}=k$. We have
\vspace{3mm}

$ \left\{\prod \limits_{i=1}^{m}X_{p_{i}}^{[b_{p_{i}k}]_{+}-\delta_{p_{i}k}},X_{j}\right\} $\\
$ =\sum \limits_{b_{p_{t}k}>0}\prod \limits_{i=1}^{t-1}X_{p_{i}}^{[b_{p_{i}k}]_{+}-\delta_{p_{i}k}}\left\{X_{p_{t}}^{[b_{p_{t}k}]_{+}},X_{j}\right\}\prod \limits_{i=t+1}^{m}X_{p_{i}}^{[b_{p_{i}k}]_{+}-\delta_{p_{i}k}}-\prod \limits_{i=1}^{p^{-1}(k)-1}X_{p_{i}}^{[b_{p_{i}k}]_{+}}\left\{X_{k}^{-1},X_{j}\right\}\prod \limits_{i=p^{-1}(k)+1}^{m}X_{p_{i}}^{[b_{p_{i}k}]_{+}} $ \\
$=\sum \limits_{b_{p_{t}k}>0}\sum \limits_{h=0}^{[b_{p_{tk}}]_{+}-1}\prod \limits_{i=1}^{t-1}X_{p_{i}}^{[b_{p_{i}k}]_{+}-\delta_{p_{i}k}}X_{p_{t}}^{h}\left\{X_{p_{t}},X_{j}\right\}X_{p_{t}}^{[b_{p_{t}k}]_{+}-h-1} \prod \limits_{i=t+1}^{m}X_{p_{i}}^{[b_{p_{i}k}]_{+}-\delta_{p_{i}k}} $\\
$.\;\;\;\;\;\;\;\;\;\;\;\;\;\; \;\;\;\;\;\;\;\;\;\;\;\;\;\;\;\;\;\;\;\;\;\;\;\;\;\;\;\;\;\;\;\;\;\;\;\;\;\;\;\;\;\; -\prod\limits_{i=1}^{p^{-1}(k)-1}X_{p_{i}}^{[b_{p_{i}k}]_{+}}X_{k}^{-1}\left\{X_{k},X_{j}\right\}X_{k}^{-1}\prod \limits_{i=p^{-1}(k)+1}^{m}X_{p_{i}}^{[b_{p_{i}k}]_{+}} $\\
$=\sum \limits_{b_{p_{t}k}>0}\sum \limits_{h=0}^{[b_{p_{tk}}]_{+}-1}\omega_{p_{t}j}\prod \limits_{i=1}^{t-1}X_{p_{i}}^{[b_{p_{i}k}]_{+}-\delta_{p_{i}k}}X_{p_{t}}^{h}X^{e_{p_{t}}+e_{j}}X_{p_{t}}^{[b_{p_{t}k}]_{+}-h-1} \prod \limits_{i=t+1}^{m}X_{p_{i}}^{[b_{p_{i}k}]_{+}-\delta_{p_{i}k}} $\\
$.\;\;\;\;\;\;\;\;\;\;\;\;\;\; \;\;\;\;\;\;\;\;\;\;\;\;\;\;\;\;\;\;\;\;\;\;\;\;\;\;\;\;\;\;\;\;\;\;\;\;\;\;\;\;\;\;
-\omega_{kj}\prod \limits_{i=1}^{p^{-1}(k)-1}X_{p_{i}}^{[b_{p_{i}k}]_{+}}X_{k}^{-1}X^{e_{k}+e_{j}}X_{k}^{-1}\prod \limits_{i=p^{-1}(k)+1}^{m}X_{p_{i}}^{[b_{p_{i}k}]_{+}} $\\
$ =(\sum \limits_{b_{p_{t}k}>0}\omega_{p_{t}j}q^{\frac{1}{2}\lambda_{jp_{t}}}\sum \limits_{h=1}^{[b_{p_{tk}}]_{+}}q^{\sum \limits_{i=t}^{m}([b_{p_{i}k}]_{+}-\delta_{p_{i}k})\lambda_{jp_{i}}-h\lambda_{jp_{t}}}-\omega_{kj}q^{\frac{1}{2}\lambda_{kj}+\sum \limits_{i=p^{-1}(k)+1}^{m}\lambda_{jp_{i}}[b_{p_{i}k}]_{+}})\prod \limits_{i=1}^{m}X_{p_{i}}^{[b_{p_{i}k}]_{+}-\delta_{p_{i}k}}. (\clubsuit) $
\vspace{3mm}

Thus we replace ($\clubsuit$) into the first term of the right-side of (\ref{*}) in the proof of Theorem \ref{pmm} and calculate as we did there, it follows that
\begin{align*}
& \left\{X_{k}^{\prime},X_{j} \right\} \\
= & (\sum \limits_{b_{p_{t}k}>0}\sum \limits_{h=1}^{[b_{p_{tk}}]_{+}}\omega_{p_{t}j}q^{\frac{1}{2}\lambda_{jp_{t}}+\sum \limits_{i=t}^{m}\lambda_{jp_{i}}([b_{p_{i}k}]_{+}-\delta_{p_{i}k})-h\lambda_{jp_{t}}}-\omega_{kj}q^{\frac{1}{2}\lambda_{kj}+\sum \limits_{i=p^{-1}(k)+1}^{m}\lambda_{jp_{i}}[b_{p_{i}k}]_{+}})X^{-e_{k}+\sum \limits_{b_{ik}>0}b_{ik}e_{i}}X_{j} \\
& +(\sum \limits_{b_{tk}<0}\sum\limits_{h=1}^{[-b_{tk}]_{+}}\omega_{tj}q^{\frac{1}{2}\lambda_{jt}+\sum
\limits_{i=t}^{m}\lambda_{ji}([-b_{ik}]_{+}-\delta_{ik})-h\lambda_{jt}}-\omega_{kj}q^{\frac{1}{2}\lambda_{kj}+\sum \limits_{i=k+1}^{m}\lambda_{ji}[-b_{ik}]_{+}})X^{-e_{k}-\sum \limits_{b_{ik}<0}b_{ik}e_{i}}X_{j}.
\end{align*}
Again, because of (\ref{omega}) and comparing coefficients of cluster Laurent monomials in $X$ due to their $\mathbb{Z}[q^{\pm \frac{1}{2}}]$-linear independence, analogue to the proof of Theorem \ref{pmm} (1), we get that

\begin{equation}\label{***}
\begin{array}{l}
\sum \limits_{b_{p_{t}k}>0}(\omega_{p_{t}j}q^{\frac{1}{2}\lambda_{jp_{t}}}\sum \limits_{h=1}^{[b_{p_{tk}}]_{+}}q^{\sum \limits_{i=t}^{m}([b_{p_{i}k}]_{+}-\delta_{p_{i}k})\lambda_{jp_{i}}-h\lambda_{jp_{t}}})-\omega_{kj}q^{\frac{1}{2}\lambda_{kj}+\sum \limits_{i=p^{-1}(k)+1}^{m}\lambda_{jp_{i}}[b_{p_{i}k}]_{+}}\\
=\sum \limits_{b_{tk}<0}(\omega_{tj}q^{\frac{1}{2}\lambda_{jt}}\sum \limits_{h=1}^{[-b_{tk}]_{+}}q^{\sum \limits_{i=t}^{m}([-b_{ik}]_{+}-\delta_{ik})\lambda_{ji}-h\lambda_{jt}})-\omega_{kj}q^{\frac{1}{2}\lambda_{kj}+\sum \limits_{i=k+1}^{m}\lambda_{ji}[-b_{ik}]_{+}}
\end{array}
\end{equation}

Now we can prove (i) in the case $b_{uk}>0$. The other case is similar.

In (\ref{***}), choose the permutation $(p_{1},\cdots,p_{m})$ to be $(1,\cdots,\hat{u},\cdots,\hat{k},\cdots,m,u,k)$ and $(1,\cdots,\hat{u},\cdots,\hat{k},$ $\cdots,m,k,u)$ respectively ($\hat{u}$ means the absence of $u$, etc.), we get two equations whose right-sides are the same as that of (\ref{***}). Subtracting these two equations, we have:
\begin{equation*}
\omega_{u j}q^{\frac{1}{2}\lambda_{ju}}\sum \limits_{h=1}^{b_{u k}}q^{(b_{u k}-h)\lambda_{ju}-\lambda_{jk}}-\omega_{k j}q^{\frac{1}{2}\lambda_{kj}}
=\omega_{u j}q^{\frac{1}{2}\lambda_{ju}}\sum \limits_{h=1}^{b_{u k}}q^{(b_{u k}-h)\lambda_{ju}}-\omega_{k j}q^{\frac{1}{2}\lambda_{kj}+b_{u k}\lambda_{ju}}
\end{equation*}
Therefore, we have
$$ \frac{\omega_{u j}}{\omega_{kj}}
=\frac{q^{\frac{1}{2}\lambda_{kj}}(1-q^{b_{u k}\lambda_{ju}})}{q^{\frac{1}{2}\lambda_{ju}}\sum \limits_{h=1}^{b_{u k}}q^{(b_{u k}-h)\lambda_{ju}}(q^{-\lambda_{jk}}-1)}
=\frac{q^{-\frac{1}{2}\lambda_{j u}}(1-q^{\lambda_{ju}})(1-q^{b_{u k}\lambda_{ju}})}{q^{-\frac{1}{2}\lambda_{jk}}(1-q^{b_{u k}\lambda_{ju}})(q_{-\lambda_{jk}}-1)}
=\frac{q^{\frac{1}{2}\lambda_{uj}}-q^{\frac{1}{2}\lambda_{ju}}}{q^{\frac{1}{2}\lambda_{kj}}-q^{\frac{1}{2}\lambda_{jk}}}. $$

Next, we prove (ii) in the case $b_{u k}b_{vk}>0$. We only proof the case $b_{u k},b_{vk}>0$, the other case is similar.

Similarly, in (\ref{***}), choose the permutation $(p_{1},\cdots,p_{m})$ to be
$(1,\cdots,\hat{u},\cdots,\hat{v},\cdots,m,u,v)$ and $(1,\cdots,\hat{u},\cdots,\hat{v},$ $\cdots,m,v,u)$
respectively, we get two equations whose right-sides are the same as that of (\ref{***}). Subtracting these two equations, we have:

$$\omega_{u j}q^{\frac{1}{2}\lambda_{ju}}\sum \limits_{h=1}^{b_{u k}}q^{(b_{u k}-h)\lambda_{ju}+b_{v k}\lambda_{jv}}+\omega_{v j}q^{\frac{1}{2}\lambda_{jv}}\sum \limits_{h=1}^{b_{vk}}q^{(b_{vk}-h)\lambda_{jv}}$$

$$=\omega_{u j}q^{\frac{1}{2}\lambda_{ju}}\sum \limits_{h=1}^{b_{u k}}q^{(b_{u k}-h)\lambda_{ju}}+\omega_{v j}q^{\frac{1}{2}\lambda_{jv}}\sum \limits_{h=1}^{b_{vk}}q^{(b_{vk}-h)\lambda_{jv}+b_{u k}\lambda_{ju}}.$$
Thus,
$$
\frac{\omega_{u j}}{\omega_{vj}}
=\frac{q^{\frac{1}{2}\lambda_{jv}}\sum \limits_{h=1}^{b_{vk}}q^{(b_{vk}-h)\lambda_{jv}}(q^{b_{u k}\lambda_{ju}}-1)}{q^{\frac{1}{2}\lambda_{ju}}\sum \limits_{h=1}^{b_{u k}}q^{(b_{u k}-h)\lambda_{ju}}(q^{b_{v k}\lambda_{jv}}-1)}
=\frac{q^{-\frac{1}{2}\lambda_{j u}}(1-q^{b_{vk}\lambda_{jv}})(q_{b_{u k}\lambda_{ju}}-1)(1-q^{\lambda_{ju}})}{q^{-\frac{1}{2}\lambda_{j v}}(1-q^{b_{u k}\lambda_{ju}})(q_{b_{v k}\lambda_{jv}}-1)(1-q^{\lambda_{jv}})}
=\frac{q^{\frac{1}{2}\lambda_{u j}}-q{\frac{1}{2}\lambda_{ju}}}{q^{\frac{1}{2}\lambda_{vj}}-q^{\frac{1}{2}\lambda_{jv}}}. $$

Moreover, consider the case for $b_{uk}b_{vk}<0$. If $\lambda_{kj}\neq 0$ or $b_{uv}\neq 0$, then we can obtain the result we want by the first part of this Lemma. If $\lambda_{kj}=b_{uv}=0$, take a mutation at direction $k$. After mutation, $\lambda^{\prime}_{uj}=\lambda_{uj}$, $\lambda_{vj}^{\prime}=\lambda_{vj}$, $\omega_{uj}^{\prime}=\omega_{uj}$ and $\omega_{vj}^{\prime}=\omega_{vj}$, but $b_{uv}^{\prime}=b_{uk}b_{kv}\neq 0$. Hence again by the first part (i) of this lemma, we have
$ \frac{\omega_{u j}}{\omega_{vj}}=\frac{q^{\frac{1}{2}\lambda_{u j}}-q{\frac{1}{2}\lambda_{ju}}}{q^{\frac{1}{2}\lambda_{vj}}-q^{\frac{1}{2}\lambda_{jv}}}. $
\end{proof}

\begin{Remark}\label{b2}
It follows from Lemma \ref{b} that
\begin{enumerate}
\item If $\lambda_{kj}\neq 0$, then $\omega_{u j}=a(q^{\frac{1}{2}\lambda_{u j}}-q^{\frac{1}{2}\lambda_{ju}})$, where $a\in \mathbb{Z}[q^{\pm \frac{1}{2}}]$ for any $b_{u k}\neq 0$.
\item If $\lambda_{vj}\neq 0$ and $b_{vk}\neq 0$, then $\omega_{u j}=a(q^{\frac{1}{2}\lambda_{u j}}-q^{\frac{1}{2}\lambda_{ju}})$, where $a\in \mathbb{Z}[q^{\pm \frac{1}{2}}]$ for any $ b_{u k}\neq 0$.
\end{enumerate}
\end{Remark}

The following theorem turns (\ref{**}) equivalently into the collection of three conditions (C$2'$), (C$3'$), (C$4'$), which are easier to deal with for us.
\begin{Theorem}\label{exchange}
If $\tilde{X}$ is log-canonical with a Poisson structure $\left\{-,-\right\}$ on a quantum cluster algebra $A_q$ and $\left\{X_{i},X_{j} \right\}=\omega_{ij}X^{e_{i}+e_{j}}$ for any $i,j\in[1,m]$, then $\mu_{k}(\tilde{X})$ is log-canonical with it if and only if the following conditions hold for any $j\in[1,m],k\in[1,n],k\neq j$:\\
\begin{itemize}
\item[(C$2'$)] For any $u\in [1,m]$, if $b_{uk}\neq0$, then $\frac{\omega_{uj}}{\omega_{kj}}=\frac{q^{\frac{1}{2}\lambda_{u j}}-q^{\frac{1}{2}\lambda_{ju}}}{q^{\frac{1}{2}\lambda_{kj}}-q^{\frac{1}{2}\lambda_{jk}}}$.
\item[(C$3'$)] For any $u,v\in [1,m]$, if $b_{uk}b_{vk}\neq 0$, then $\frac{\omega_{uj}}{\omega_{vj}}=\frac{q^{\frac{1}{2}\lambda_{uj}}-q^{\frac{1}{2}\lambda_{ju}}}{q^{\frac{1}{2}\lambda_{vj}}- q^{\frac{1}{2}\lambda_{jv}}}$.
\item[(C$4'$)] $\sum\limits_{t:\lambda_{tj}=0}\omega_{tj}b_{tk}=0$.
\end{itemize}
\end{Theorem}
\begin{proof}
For the necessary part, Lemma \ref{b} claim the first two conditions. Combining these with equations (\ref{**}), we can reach the third one case by case:

Case 1: $\lambda_{kj}\neq 0$. Then $\omega_{tj}=0$ if $b_{tk}\neq 0$ and $\lambda_{tj}=0$, thus $\sum\limits_{\lambda_{tj}=0}\omega_{tj}b_{tk}=0$.

Case 2: $\lambda_{kj}= 0$ and there are $u,v\in[1,m]$ such that $b_{uk}>0,b_{vk}<0,\lambda_{uj}\lambda_{vj}\neq 0$. Like case 1, $\omega_{tj}=0$ if $b_{tk}\neq 0$ and $\lambda_{tj}=0$.

Case 3: $\lambda_{kj}= 0$, there is $u$ such that $b_{uk}>0,\lambda_{uj}\neq 0$ and for any $v$ such that $b_{vk}<0$, $\lambda_{vj}=0$. Then $\omega_{tj}=0$ if $b_{tk}>0$ and $\lambda_{tj}=0$ and equations (\ref{**}) can be simplified as
\[a_{u}(1-q^{\sum\limits_{t=1}^{m}[b_{tk}]_{+}\lambda_{jt}})=\sum\limits_{t}\omega_{tj}[-b_{tk}]_{+}\]
Because $\sum\limits_{t=1}^{m}[b_{tk}]_{+}\lambda_{jt}=\sum\limits_{t}b_{tk}\lambda_{jt}=0$, we have $\sum\limits_{t}\omega_{tj}[-b_{tk}]_{+}=0$. Therefore $\sum\limits_{\lambda_{tj}=0}\omega_{tj}b_{tk}=0$.

Case 4: $\lambda_{kj}= 0$, there is $v$ such that $b_{vk}<0,\lambda_{vj}\neq 0$ and for any $u$ such that $b_{uk}>0$, $\lambda_{uj}=0$. Similar to case 3.

Case 5: $\lambda_{kj}= 0$ and for any $u$ such that $b_{uk}\neq 0$, $\lambda_{uj}=0$. Then (\ref{**}) looks like
\[\sum\limits_{t}\omega_{tj}[b_{tk}]_{+}=\sum\limits_{t}\omega_{tj}[-b_{tk}]_{+}\]
thus $\sum\limits_{\lambda_{tj}=0}\omega_{tj}b_{tk}=0$.

The sufficent part can be seen by direct calculations. Once these conditions are true, the formula (\ref{**}) holds.
\end{proof}
\vspace{4mm}

\section{Philosophy of second quantization for quantum cluster algebras}

For the family of Poisson matrices $\Psi(t)$ ($t\in \mathbb T_n$) of a Poisson structure on a cluster algebra $A$, it is interesting to note that their mutation formula in (\ref{non-quantum omega mutation}) is the same as that in (\ref{lambdamutation}) for the deformation matrices $\Lambda(t)$ of a quantum cluster algebra $A_q$. From this fact, the relation between the Poisson structures of a cluster algebra and the quantization of this algebra can be given below.

For a quantum cluster algebra $A_q=A_{q}(\Sigma)$ with seeds $\Sigma(t)=(X(t),\tilde{B}(t),\Lambda(t))$ at $t\in \mathbb{T}_{n}$, whose CCV is the cluster algebra $A$, let $\Psi(t)=\Lambda(t)$, then we obtain a compatible Poisson structure on $A$ with Poisson matrices $\Psi(t)$. This compatible Poisson structure on $A$ is given by the family of the deformation matrices $\Lambda(t)$ of $A_q$.

Conversely, assume $\Psi(t)$ ($t\in \mathbb{T}_n$) are the family of Poisson matrices of a compatible Poisson structure of a cluster algebra $A$. Let $\Lambda(t)=\Psi(t)$, according to Theorem 3.2 (i) and (iii) in \cite{N}, $(B(t), \Lambda(t))$ is a compatible pair satisfying the condition (\ref{2}). Then we obtain a quantum cluster algebra $A_q$ as the CQV of $A$.

By the above discussion, we have the following statement:
\begin{Observation} Assume $ A$ is a cluster algebra with exchange matrices $\tilde{B}(t)$ which are of full column rank. Then, we have the following one-by-one correspondence:

.\;\;$\{$Compatible Poisson structures of $A$$\}$\;

.\;\;\;\;\;\;\;\;\;\;\;\;\;\;\;\;\;\;\;\;\;\;\;\;\;\;\;\;\;\;\;\;\;\;\;\;\;$\longleftrightarrow$\; $\{$Quantizations of $A\}=\{$quantum cluster algebras as CQV's of $A\}$\\
via

.\;\;$\{$Poisson matrices of $A \}\;=\;\{$Deformation matrices of $A_q \}$.

\end{Observation}
Motivated by this observation, for a quantum cluster algebra $A_q$, if we have a (non-trivial) compatible Poisson structure on it, would it be possible to find an algebra as a further quantization of $A_q$ whose {\em like-compatible pairs} are correspondent to the family of Poisson matrices of the Poisson structure of $A_q$? We will call this possible further quantization of $A_q$ as the {\em second quantization} of $A_q$.

This is the reason we want to find out the (non-trivial) compatible Poisson structure on a quantum cluster algebra $A_q$.

In this section, we would like to give the exact definition of the so-called second quantization of $A_q$. \vspace{2mm}

We introduce the $q$-analog of an integer $a$ which is $[a]_{q}=\frac{q^{a}-q^{-a}}{q-q^{-1}}\in\mathbb{N}(q^{\pm 1})$ for $q\in\mathbb C$. Given a deformation matrix $\Lambda$ and a Poisson matrix $\Omega$, we can define an $m\times m$ skew-symmetric matrix $W(t)=(W_{ij})$ as
\begin{equation}\label{IImatrix}
W_{ij}=\left\{
\begin{array}{ll}
\frac{\omega_{ij}\lambda_{ij}}{[\lambda_{ij}]_{q^{\frac{1}{2}}}} & \lambda_{ij}\neq 0\\
\omega_{ij} & \lambda_{ij}=0.
\end{array}
\right .
\end{equation}
We call $W(t)$ the {\bf second deformation matrix} at $t$.

From this definition we know that any two of $\Omega(t)$, $\Lambda(t)$ and $W(t)$ can determine the other one. And the three conditions in Theorem \ref{exchange} can also be restated for any $j\in[1,m],k\in[1,n],k\neq j$ as following:
\begin{itemize}
\item[(C2)] For any $u\in [1,m]$, if $b_{uk}\neq0$, then $\frac{W_{uj}}{W_{kj}}=\frac{\lambda_{u j}}{\lambda_{kj}}$.
\item[(C3)] For any $u,v\in [1,m]$, if $b_{uk}b_{vk}\neq 0$, then $\frac{W_{uj}}{W_{vj}}=\frac{\lambda_{uj}}{\lambda_{vj}}$.
\item[(C4)] $\sum\limits_{t:\lambda_{tj}=0}W_{tj}b_{tk}=0$.
\end{itemize}

\begin{Definition}\label{triple}
For $t\in\mathbb T_n$, let $\tilde{B}(t)=(b_{ij})$ be an $m\times n$ integer matrix with $m\geqslant n$, $\Lambda(t)=(\lambda_{ij})$ an $m\times m$ integer skew-symmetric matrix and $W(t)=(W_{ij})$ an $m\times m$ skew-symmetric integer matrix. The triple $(\tilde{B}(t),\Lambda(t),W(t))$ is called \textbf{compatible} if $(\tilde{B}(t),\Lambda(t))$ is a compatible pair satisfying $(C1)$ and any triple mutation equivalent to $(\tilde{B}(t),\Lambda(t),W(t))$ satisfies $(C2),(C3)$ and $(C4)$.
\end{Definition}
Recall that by Theorem \ref{exchange}, the latter condition including $(C2),(C3)$ and $(C4)$ is equivalent to that the Poisson structure induced by $\Omega(t)$ is compatible with the quantum cluster algebra $A_q$ associated to the compatible pair $(\tilde{B}(t),\Lambda(t))$.

As usual we define the {\bf extended cluster} at $t\in\mathbb{T}_{n}$ to be a set of variables
\begin{equation*}
\tilde{Y}(t)=\left\{Y_{t}^{e_{1}},Y_{t}^{e_{2}},\cdots,Y_{t}^{e_{n}},Y^{e_{n+1}},\cdots,Y^{e_{m}}\right\},
\end{equation*}
where $e_{i}\in \mathbb{Z}^{m}$ are the standard basis. And the set of first $n$ variables is called the \textbf{cluster} at $t$ and denoted by $Y(t)$. For $p,q\in\mathbb C$, let $\mathcal{T}_{t}$ be the $\mathbb{Z}[p^{\pm\frac{1}{2}},q^{\pm\frac{1}{2}}]$-algebra generated by $\tilde{Y}(t)$ satisfying the relation
\begin{equation}\label{second quantum exchange}
Y_{t}^{e_{i}}Y_{t}^{e_{j}}=p^{\frac{1}{2}W_{ij}}q^{\frac{1}{2}\lambda_{ij}}Y_{t}^{e_{i}+e_{j}},\forall i,j\in[1,m],
\end{equation}
We call $\mathcal T_t$ the {\bf II-quantum torus}, or say, {\bf $(p,q)$-quantum torus at $t$}.

Denote by $\mathcal{F}_{p,q}$ the skew-field of fractions of $\mathcal{T}_{t}$. Thus, $\mathcal{T}_t$ is a subalgebra of $\mathcal F_{p,q}$.

We can see that
\[ Y_{t}^{e_{i}}Y_{t}^{e_{j}}=p^{W_{ij}}q^{\lambda_{ij}}Y_{t}^{e_{j}} Y_{t}^{e_{i}},\forall i,j\in[1,m].\]
We call $\Sigma(t)=(\tilde{Y}(t),\tilde{B}(t),\Lambda(t),W(t))$ a {\bf II-quantum seed} at $t$ for the compatible triple $(\tilde{B}(t),\Lambda(t),W(t))$.
\begin{Definition}\label{IIm}
Let $\Sigma(t)$ and $\Sigma(t^{\prime})$ be two II-quantum seeds at $t$ and $t^{\prime}$ respectively. Denote by $b_{i}$ the i-column of $\tilde{B}(t)$ as a vector. Assume $t$ and $t^{\prime}$ are adjacent vertices by an edge labeled $k$ in $\mathbb{T}_{n}$. $\Sigma(t^{\prime})$ is defined from $\Sigma(t)$ by a \textbf{mutation in direction $k$} if $\Sigma(t^{\prime})=\mu_{k}(\Sigma(t))=(\mu_{k}(\tilde{Y}(t)),\mu_{k}(\tilde{B}(t)),\mu_{k}(\Lambda(t)),\mu_{k}(W(t)))$, where
\begin{equation*}
\tilde{Y}(t')=\mu_{k}(\tilde{Y}(t))=(\tilde{Y}(t)\setminus\{Y_{t}^{e_{k}}\})\bigcup\{\mu_{k}(Y_{t}^{e_{k}})\},\quad \; Y_{t'}^{e_{k}}=\mu_{k}(Y_{t}^{e_{k}})=Y_{t}^{-e_{k}+[b_{k}(t)]_{+}}+Y_{t}^{-e_{k}+[-b_{k}(t)]_{+}}
\end{equation*}
and
\begin{equation}
W_{ij}(t^{\prime})=\left\{
\begin{array}{lcr}
-W_{kj}(t)+\sum\limits_{l=1}^{m}[b_{lk}(t)]_{+}W_{lj}(t)&&\text{if}\; i=k\neq j\\
-W_{ik}(t)+\sum\limits_{l=1}^{m}[b_{lk}(t)]_{+}W_{il}(t)&&\text{if}\; j=k\neq i\\
W_{ij}(t)&& \text{otherwise}
\end{array}
\right .
\end{equation}
while the mutations of matrices $\tilde{B}(t)$ and $\Lambda(t)$ are the same as those we introduced before in (\ref{matrixmut}) and (\ref{lambdamutation}) respectively.
\end{Definition}

\begin{Theorem}
For any $t\in\mathbb{T}_{n}$ and $k\in[1,n]$, let $\Sigma(t)=(\tilde{Y},\tilde{B},\Lambda,W)$ be a II-quantum seed at $t$, then $\mu_{k}(\Sigma(t))=(\tilde{Y}^{\prime},\tilde{B}^{\prime},\Lambda^{\prime},W^{\prime})$ is also a II-quantum seed. And the Poisson structure associated to $\Sigma(t)$ and $\mu_{k}(\Sigma(t))$ are the same.

\end{Theorem}
\begin{proof}
The compatibility of $(\tilde{B},\Lambda,W)$ is mutation invariant by definition, so what left is to prove
\[{Y^{\prime}}^{e_{k}}Y^{e_{j}}=p^{W^{\prime}_{kj}}q^{\lambda^{\prime}_{kj}}Y^{e_{j}}{Y^{\prime}}^{e_{k}},\forall j\in[1,m]\setminus\{k\}.\]

Since ${Y^{\prime}}^{e_{k}}=Y^{-e_{k}+[b_{k}]_{+}}+Y^{-e_{k}+[-b_{k}]_{+}}$, we have
\begin{equation*}
\begin{array}{rl}
{Y^{\prime}}^{e_{k}}Y^{e_{j}}= & (Y^{-e_{k}+[b_{k}]_{+}}+Y^{-e_{k}+[-b_{k}]_{+}})Y^{e_{j}} \\
= & p^{\sum\limits_{i=1}^{m}[b_{ik}]_{+}W_{ij}-W_{kj}}q^{\sum\limits_{i=1}^{m}[b_{ik}]_{+}\lambda_{ij}-\lambda_{kj}}Y^{e_{j}}Y^{-e_{k}+[b_{k}]_{+}} \\
& +p^{\sum\limits_{i=1}^{m}[b_{ik}]_{-}W_{ij}-W_{kj}}q^{\sum\limits_{i=1}^{m}[b_{ik}]_{-}\lambda_{ij}-\lambda_{kj}}Y^{e_{j}}Y^{-e_{k}+[b_{k}]_{-}} \\
= & p^{W^{\prime}_{kj}}q^{\lambda^{\prime}_{kj}}Y^{e_{j}}{Y^{\prime}}^{e_{k}}.
\end{array}
\end{equation*}

Let $\Omega$ be the Poisson matrix associated to $\Sigma(t)$. We need to verify that the Poisson matrix $\Omega^{\prime}$ associated to $\mu_{k}(\Sigma(t))$ is exactly the Poisson matrix $\Omega^{\prime\prime}$ obtained from $\Omega$ by mutation at direction $k$. If $i,j\neq k$, then $\omega_{ij}^{\prime\prime}=\omega_{ij}=\omega_{ij}^{\prime}$. Next assume $i=k\neq j$, the case $i\neq k=j$ is the same.

If any $l\in S=\left\{u\mid u=k \text{ or } b_{uk}\neq 0\right\}$, $\lambda_{lj}=0$, we have $\lambda_{kj}^{\prime}=-\lambda_{kj}+\sum\limits_{l=1}^{m}[b_{lk}]_{+}\lambda_{lj}=0$. Then $W_{ij}=\omega_{ij}$ for any $l\in S$ and $W_{kj}^{\prime}=W_{kj}$. Therefore following (\ref{omega mutation formula})
\[\omega^{\prime\prime}_{kj}=-\omega_{kj}+\sum\limits_{l=1}^{m}[b_{lk}]_{+}\omega_{lj}=-W_{kj}+\sum\limits_{l=1}^{m}[b_{lk}]_{+}W_{lj}=W_{kj}^{\prime}=W_{kj}=\omega_{kj}.\]

If there is $v\in S$ such that $\lambda_{vj}\neq 0$, then by Remark \ref{b2} $\omega_{ij}=a[\lambda_{ij}]_{q^{\frac{1}{2}}}$ for any $i\in S$, where $a\in \mathbb{Z}$. So $W_{ij}=a\lambda_{ij}$ and $W_{kj}^{\prime}=a\lambda_{ij}$. Again by (\ref{omega mutation formula}), it can be checked that
\[\omega_{kj}^{\prime\prime}=a[\lambda_{kj}^{\prime}]_{q^{\frac{1}{2}}},\]
hence $W_{kj}^{\prime\prime}=a\lambda_{kj}^{\prime}=W_{kj}^{\prime}$.

Therefore $\Omega^{\prime}=\Omega^{\prime\prime}$, which means the Poisson structures induced by $\Omega$ and by $\Omega^{\prime}$ are the same.
\end{proof}

\begin{Definition} \label{IIq} For a quantum cluster algebra $A_q$ with a compatible Poisson structure $\{\;,\;\}$,
assign II-quantum seeds $\Sigma(t)$ to every vertex $t$ in $\mathbb{T}_{n}$ so that for any $t$ and $t^{\prime}$ adjacent by an edge labeled $k$, $\Sigma(t^{\prime})$ is obtained from $\Sigma(t)$ by a mutation in direction $k$ by Definition \ref{IIm}. Denote by $A_{p,q}=A_{p,q}(\Sigma)$ the $\mathbb{Z}[p^{\pm\frac{1}{2}},q^{\pm\frac{1}{2}}][Y^{\pm e_{n+1}},\cdots,Y^{\pm e_{m}}]$-subalgebra of $\mathcal{F}_{p,q}$ generated by $\bigcup\limits_{t\in\mathbb{T}_{n}}Y(t)$. We call $A_{p,q}$ the \textbf{secondly quantized cluster algebra} associated to $\{\Sigma(t)\}_{t\in\mathbb{T}_{n}}$, or say, the {\bf second quantization} of $A_q$.
\end{Definition}

Trivially, if $p$ tends to 1 or $q$ tends to 1 or $p$ tends to $q$, then the secondly quantized cluster algebra $A_{p,q}$ degenerates to the quantum cluster algebras $A_q$, $A_p$, $A_q$ with deformation matrix $\Lambda$, $W$, $\Lambda+W$ respectively.

As we said before, any two of $\Omega(t)$, $\Lambda(t)$ and $W(t)$ determine the other one.
Hence, if we are given a secondly quantized cluster algebra $A_{p,q}$ with deformation matrices $\Lambda(t)=(\lambda_{i,j})$ and second deformation matrices $W(t)=(W_{ij})$, then we can obtain the Poisson matrices $\Omega(t)$ of a compatible Poisson structure on $A_q$ via:
\begin{equation*}
\omega_{ij}=\left\{
\begin{array}{ll}
\frac{W_{ij}[\lambda_{ij}]_{q^{\frac{1}{2}}}}{\lambda_{ij}}, & \lambda_{ij}\neq 0\\
W_{ij}, & \lambda_{ij}=0.
\end{array}
\right.
\end{equation*}
Therefore, when $\Lambda(t)$ is fixed, we have the following correspondance:
\begin{Observation} Assume $A_q$ is a quantum cluster algebra with the compatible pairs $(\tilde{B}(t),\Lambda(t))$. Then, we have the following one-by-one correspondence:
\\
$\{$Compatible Poisson structures of $A_q\}$\;
\\
.\;\;\;\;\;\;\;\;\;\;\;\;\;\;\;\;\;\;\;$\longleftrightarrow$\; $\{$Second Quantizations of $A_q\}=\{$Secondly quantized cluster algebras $A_{p,q}$ of $A_q\}$\\
via

.\;\;$\{$Poisson matrices of $A_q \}\;\; \longleftrightarrow\;\;\{$Second deformation matrices of $A_{p,q} \}$.

\end{Observation}

We say a matrix $B$ has a matrix decomposition $B=B_{1}\bigoplus B_{2}$ if
\[B=\begin{pmatrix}
B_{1} & O \\
O & B_{2}
\end{pmatrix}\]

For a quantum cluster algebra $A_{q}$, if for a certain seed $(\tilde{Y},\tilde{B},\Lambda,W)$, there are the matrix decompositions $W=\bigoplus\limits_{i=1}^{r}W_{i}$ and $\Lambda=\bigoplus\limits_{i=1}^{r}\Lambda_{i}$ such that
\begin{equation}\label{mb}
W_{i}=a_{i}\Lambda_{i},\; \text{ for some }\; a_{i}\in\mathbb{Z}\;\; \forall i=1,\cdots,r,
\end{equation}
and let $I_{i}$ be the index set of $W_{i}$, then

(i)\[X_{k}X_{l}=\left\{\begin{array}{lc}
p^{\frac{1}{2}W_{kl}}q^{\frac{1}{2}\lambda_{kl}}X_{k}X_{l}=(p^{a_{i}}q)^{\frac{1}{2}\lambda_{kl}}X_{k}X_{l} & k,l\in I_{i} \\
X_{l}X_{k} & k\in I_{i},l\in I_{j},i\neq j
\end{array}\right.\]

(ii)in this case, (\ref{mb}) and (i) always hold for any seed.

Hence, the secondly quantized cluster algebra $A_{p,q}$ is essentially a quantum cluster algebra. Thus under this condition, we say the secondly quantized cluster algebra $A_{p,q}$ to be {\bf trivial}. Otherwise, it is called {\bf non-trivial}.

For example $W=a\Lambda,a\in\mathbb{Z}$, then there is a canonical $\mathbb{Z}[p^{\pm\frac{1}{2}}]$-algebra isomorphism $A_{p,q}\cong A_{q}\otimes_{\mathbb{Z}[q^{\pm\frac{1}{2}}]}\mathbb{Z}[p^{\pm\frac{1}{2}},q^{\pm\frac{1}{2}}]$ sending $Y^{e_{i}}_{t}$ to $X^{e_{i}}_{t}$, $q$ to $p^{-a}q$ and $p$ to $p$, therefore the quantum cluster algebra $A_{q}$ can be embedded into the secondly quantized cluster algebra $A_{p,q}$, which means that their cluster algebraic structures coincide.

\begin{Proposition}\label{BW=D}
Let $(\tilde{X}(t),\tilde{B}(t),\Lambda(t))$ be a seed of a quantum cluster algebra $A_{q}$ at $t\in \mathbb{T}_n$ and $\left\{-,-\right\}$ a compatible Poisson structure on $A_{q}$. Then the second deformation matrix $W(t)$ satisfies that
\begin{equation}\label{add}
\tilde{B}(t)^{\top}W(t)=c(D\; O),
\tag{C1$^*$}
\end{equation}
where $c\in\mathbb{Z}[q^{\pm\frac{1}{2}}]$ and $D$ is the skew-symmetrizer of $\tilde{B}(t)$.
\end{Proposition}
\begin{proof}
Here, since we only discuss with the seed of $A_q$ at the vertex $t$, we will omit $t$ for clusters, exchange matrices, deformation matrices and etc.. Let $\tilde{B}_{l} (l\in [1,n])$ denote the $l$-th column of $\tilde{B}$.
Let $k\in[1,n],j\in[1,m]$ and $k\neq j$.

If $\lambda_{kj}\neq 0$ or $b_{lk}\lambda_{lj}\neq 0$ for some $l$, then according to Remark \ref{b2}, $\omega_{ij}=a[\lambda_{ij}]_{q^{\frac{1}{2}}}$ for any $i$ such that $b_{ik}\neq 0$ or $i=k$, where $a\in\mathbb{Z}[q^{\pm\frac{1}{2}}]$. Hence $w_{ij}=a\lambda_{ij}$ for any $i$ such that $b_{ik}\neq 0$ or $i=k$ and $\tilde{B}_{k}^{\top}W_{j}=\sum\limits_{i=1}^{m}b_{ik}W_{ij}=a\sum\limits_{i=1}^{m}b_{ik}\lambda_{ij}=0$, where $W_{j}$ represents the $j$-th column of $W$.

Otherwise, if $\lambda_{ij}=0$ for all $i$ such that $b_{ik}\neq 0$ or $i=k$. Then we know from Lemma \ref{exchange} that $\tilde{B}_{k}^{\top}W_{j}=\sum\limits_{i=1}^{m}b_{ik}W_{ij}=\sum\limits_{i=1}^{m}b_{ik}\omega_{ij}=0$.

Therefore, $\tilde{B}^{\top}W=(M~ O)$, where $M$ is an $n\times n$ diagonal matrix. Moreover, $\tilde{B}^{\top}W\tilde{B}$ is a skew-symmetric matrix, which means $MB$ is skew-symmetric, i.e. $M=cD$.

\end{proof}

Hence given a compatible triple $(\tilde{B},\Lambda,W)$, $(\tilde{B},W)$ is also a compatible pair in the meaning of the condition (C1). Moreover, by Definition \ref{triple}, it can be checked that $(\tilde{B},W,a\Lambda)$ is a compatible triple for certain $a\in\mathbb{Z}$ to make sure the corresponding Poisson matrix having elements in $\mathbb{Z}[q^{\pm\frac{1}{2}}]$. Therefore, $(\tilde{B},\Lambda)$ and $(\tilde{B},W)$ induce two quantization $A_{q}$ and $A^{\prime}_{p}$ for a cluster algebra $A$ associated to $\tilde{B}$. In general, $A_q$ and $A^{\prime}_p$ are different as quantum cluster algebras. Then $(\tilde{B},\Lambda,W)$ and $(\tilde{B},W,a\Lambda)$ respectively induces a second quantization $A_{p,q}$ and $A_{q,p}^{\prime}$ for $A_{q}$ and $A^{\prime}_{p}$. However, according to the definition of secondly quantized algebras, there is a cluster isomorphism $\psi$ from $A_{p,q}\otimes_{\mathbb{Z}[q^{\pm\frac{1}{2}}]}\mathbb{Z}[q^{\pm\frac{1}{2a}}]$ to $A_{q,p}^{\prime}$ fixing variables and sending $p$ to $q$, $q$ to $p^{a}$.

We will call the quantum cluster algebra $A^{\prime}_p$ a {\bf dual quantum cluster algebra} of $A_q$ on the compatible pair $(\tilde{B},\Lambda,W)$.

Therefore, we have the following two ways of quantization induced by a triple $(\tilde{B},\Lambda,W)$:
\begin{figure}[H]
\centering
\includegraphics[width=100mm]{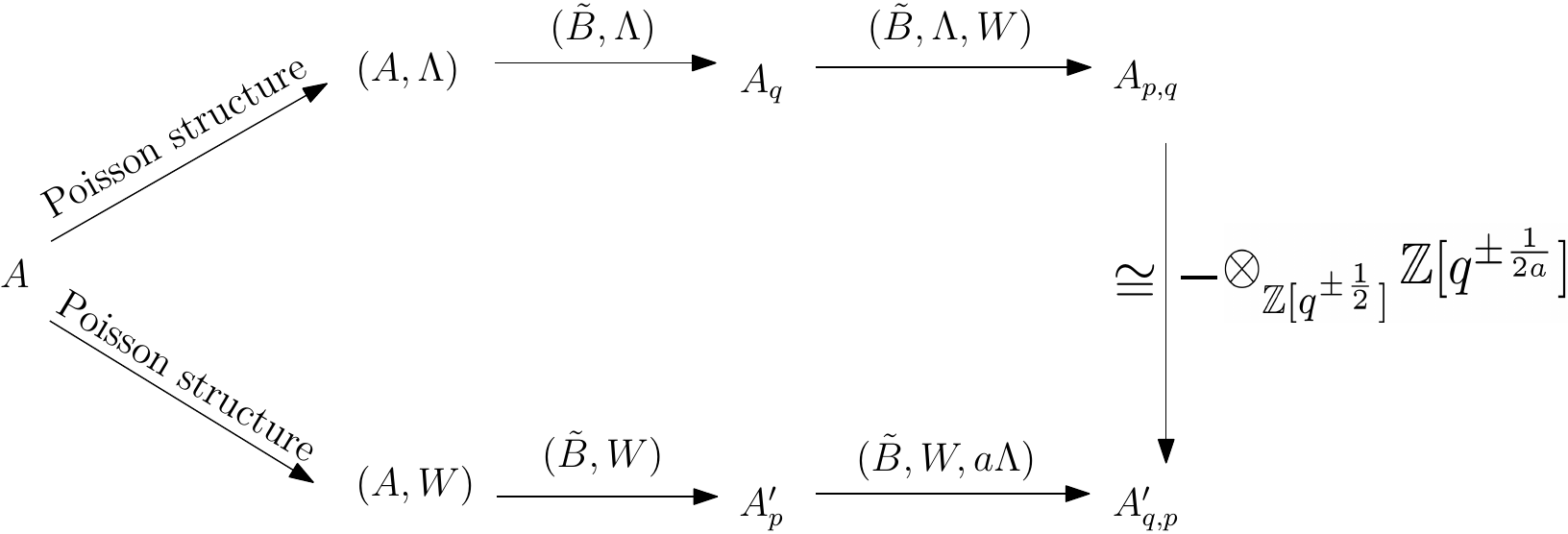}
\caption{Two ways of quantization}
\end{figure}

In summary, we have the following conclusion:
\begin{Theorem}\label{twoways}
For a cluster algebra $A$ and a compatible triple $(\tilde{B},\Lambda,W)$ from its two various Poisson structures, the secondly quantized cluster algebras $A_{p.q}$ and $A'_{q,p}$ from $A_q$ and its dual $A'_p$ determined by the compatible pairs $(\tilde{B},\Lambda)$ and $(\tilde{B},W)$ respectively with the compatible triples $(\tilde{B},\Lambda,W)$ and $(\tilde{B},W,a\Lambda)\; (a\in\mathds{Z})$ are essentially the same under the cluster isomorphism $\psi$:
$$\psi:\; A_{p,q}\otimes_{\mathbb{Z}[q^{\pm\frac{1}{2}}]}\mathbb{Z}[q^{\pm\frac{1}{2a}}]\longrightarrow A_{q,p}^{\prime}$$
fixing variables and sending $p$ to $q$, $q$ to $p^{a}$.
\end{Theorem}

The basic idea of this paper is to lift the Poisson matrices of (quantum) cluster algebras to the deformation matrices, so as to complete the first or second quantization of this algebra. Based on this idea, we can similarly continue to obtain the third quantization of a secondly quantized cluster algebra, even in general, to obtain the $(n+1)$-th quantization of an $n$-th quantized cluster algebra. The whole progress would be much like that in this paper but with more parameters. The further quantization may go on until it is faced to no non-trivial quantization, since all matrices $R$ satisfying $\tilde{B}^{\top}R=cD$ consists of a finite-dimension linear space.

Particularly, in Section 6, we will prove that the second quantization is trivial in the case without coefficients but may be admitted to be non-trivial in the case with coefficients.

Recently, before this work (the first version) has be finished in the beginning of this year, we found {\em toroidal cluster algebras} defined in \cite{GY} and \cite{FH}. According to Proposition \ref{BW=D}, a secondly quantized cluster algebra in this paper is a toroidal cluster algebra associated to an initial seed containing a compatible triple.

In \cite{FH}, Laurent phenomena was proved to be true for toroidal cluster algebras in skew-symmetrizable case as well as positivity for cluster variable in skew-symmetric case. Moreover, the exchange graph of a toroidal cluster algebra is incident to that of its corresponding (quantum) cluster algebra. Therefore, the corresponding properties of secondly quantized cluster algebras naturally follow as they are a special kind of toroidal cluster algebras.
\vspace{4mm}

\section{A lemma for decomposability of quantum cluster algebras}
First of all, let us focus on the condition (C4), which looks similar to the condition for a Poisson bracket being compatible with a cluster algebra in \cite{N}. Note that in (C$4'$): $\sum\limits_{\lambda_{tj}=0}\omega_{tj}b_{tk}=0$, the sum is given only for those terms whose corresponding $\lambda_{tj}=0$. So, we define a new matrix $\hat{\Omega}=\{\hat{\omega}_{ij}\}$ from $\Omega$ where $\hat{\omega}_{ij}=\omega_{ij}$ if $\lambda_{ij}=0$ and $\hat{\omega}_{ij}=0$ otherwise.

For any $i\neq j\in[1,m]$, we say that $i$ {\bf is connected to} $j$ (with respect to $B$) if $b_{ij}\neq 0$. We define the relation $i\sim j$ if either $i=j$ or $i$ is connected to $j$.

\begin{Lemma}\label{zero}
In a quantum cluster algebra $A_q$ with a Poisson structure $\{-,-\}$, let $(\tilde{X}^{\prime},\tilde{B}^{\prime},\Lambda^{\prime})$ be a seed mutation equivalent to the initial seed $(\tilde{X},\tilde{B},\Lambda)$. Assume $\tilde{X}$ and $\tilde{X}^{\prime}$ are log-canonical with respective to $\{-,-\}$ with Poisson matrices $\Omega$ and $\Omega^{\prime}$ respectively. Then $\sum\limits_{\lambda_{tj}^{\prime}=0}\omega_{tj}^{\prime}b_{tk}^{\prime}=0$ for any $j\neq k$ if and only if $\hat{\Omega} \tilde{B}=c\tilde{D}$, where c$\in \mathbb{Z}[q^{\pm\frac{1}{2}}]$,
$\tilde{D}=
\begin{pmatrix}
D\\
O
\end{pmatrix}_{m\times n}$ for $D$
a skew-symmetrizer of $B$. And, in this case $\hat{\Omega}^{\prime} \tilde{B}^{\prime}=c\tilde{D}$.
\end{Lemma}
\begin{proof}
This proof is similar to that in \cite{N}, but more annoying.

First assume $(\tilde{X}^{\prime},\tilde{B}^{\prime},\Lambda^{\prime})=\mu_{k}(\tilde{X},\tilde{B},\Lambda)$. For any $j\in[1,m],h\in[1,n],j\neq h$, we write $\sum\limits_{\lambda_{tj}=0}\omega_{tj}b_{th}=0$ and $\sum\limits_{\lambda_{th}=0}\omega_{th}b_{th}=a_{h}$. We will show that $\sum\limits_{\lambda_{tj}^{\prime}=0}\omega_{tj}^{\prime}b_{th}^{\prime}=0$ for $j\neq h$ if and only if $\hat{\Omega}\tilde{B}=c\tilde{D}$, and then in this case $\hat{\Omega}^{\prime}\tilde{B}^{\prime}=c\tilde{D}$. Then by induction this lemma holds for any cluster mutation equivalent to the initial cluster.

First We prove that $\sum\limits_{\lambda_{tj}^{\prime}=0}\omega_{tj}^{\prime}b_{th}^{\prime}=0$ for $j\neq h$ if and only if $\hat{\Omega}\tilde{B}=c\tilde{D}$ case by case.

Case 1: $k\neq h,j$ and there is $l\sim k$ such that $\lambda_{lj}\neq 0$. Because $\lambda_{lj}^{\prime}=\lambda_{lj}\neq 0$ and $b_{lk}^{\prime}=-b_{lk}\neq 0$, we have $\omega_{tj}=0$ if $\lambda_{tj}=0$ and $t\sim k$, and $\omega_{kj}^{\prime}=0$ if $\lambda_{kj}^{\prime}=0$ according to Lemma \ref{b}. Hence we have that
$$\sum\limits_{\lambda_{tj}^{\prime}=0}\omega_{tj}^{\prime}b_{th}^{\prime}
=\sum_{\mbox{\tiny$\begin{array}{c}\lambda_{tj}^{\prime}=0 \\t\neq k\end{array}$}}\omega_{tj}^{\prime}b_{th}^{\prime}
=\sum_{\mbox{\tiny$\begin{array}{c}\lambda_{tj}=0 \\t\neq k\end{array}$}}\omega_{tj}^{\prime}b_{th}^{\prime}
=\sum_{\mbox{\tiny$\begin{array}{c}\lambda_{tj}=0 \\t\neq k\end{array}$}}\omega_{tj}(b_{th}+[b_{tk}]_{+}b_{kh}+b_{tk}[-b_{kh}]_{+})$$
$$ =\sum\limits_{\lambda_{tj}=0}\omega_{tj}b_{th}+\sum\limits_{\lambda_{tj}=0}\omega_{tj}b_{kh}[b_{tk}]_{+}+\sum\limits_{\lambda_{tj}=0}
\omega_{tj}b_{tk}[b_{kh}]_{+}
=\sum\limits_{\lambda_{tj}=0}\omega_{tj}b_{kh}[b_{tk}]_{+}
=0.$$
\vspace{3mm}

Case 2: $k\neq h,j$ and for any $l\sim k,\lambda_{lj}=0$. In this case, we can calculate that the mutation formula of $\Omega$ is
\begin{equation*}
\omega_{ij}^{\prime}=\left\{
\begin{array}{lcr}
\omega_{ij}&&i\neq k\\
-\omega_{kj}+\sum\limits_{t=1}^{m}[b_{tk}]_{+}\omega_{tj}&&i=k
\end{array}
\right .
\end{equation*}
So we have
\begin{flalign*}
&\sum\limits_{\lambda_{tj}^{\prime}=0}\omega_{tj}^{\prime}b_{th}^{\prime}
\;=\;\omega_{kj}^{\prime}b_{kh}^{\prime}+\sum_{\mbox{\tiny$\begin{array}{c}\lambda_{tj}^{\prime}=0 \\t\neq k\end{array}$}} \omega_{tj}^{\prime}b_{th}^{\prime}\\
=\;&(-\omega_{kj}+\sum\limits_{t=1}^{m}[b_{tk}]_{+}\omega_{tj})(-b_{kh})+\sum_{\mbox{\tiny$\begin{array}{c}\lambda_{tj}=0 \\t\neq k\end{array}$}}\omega_{tj}(b_{th}+[b_{tk}]_{+}b_{kh}+b_{tk}[-b_{kh}]_{+})\\
=&\sum\limits_{\lambda_{tj}=0}\omega_{tj}b_{th}+[-b_{kh}]_{+}\sum\limits_{\lambda_{tj}=0}\omega_{tj}b_{tk}\\
=&0.
\end{flalign*}
\vspace{3mm}

Case 3: $k=h$. We have
\[\sum\limits_{\lambda_{tj}^{\prime}=0}\omega_{tj}^{\prime}b_{th}^{\prime}
=\sum_{\mbox{\tiny$\begin{array}{c}\lambda_{tj}^{\prime}=0 \\t\neq h\end{array}$}}\omega_{tj}^{\prime}b_{th}^{\prime}
=\sum_{\mbox{\tiny$\begin{array}{c}\lambda_{tj}=0 \\t\neq h\end{array}$}}\omega_{tj}(-b_{th})
=-\sum\limits_{\lambda_{tj}=0}\omega_{tj}b_{th}
=0.\]

Case 4: $k=j$. For $i\in[1,m]$, if there is $l\sim j$ so that $\lambda_{il}\neq 0$, then $\lambda_{ij}^{\prime}=0$ induces $\omega_{ij}^{\prime}=0$. Otherwise if $\lambda_{il}=0$ for any $l\sim j$, then $\omega_{ij}^{\prime}=-\omega_{ij}+\sum\limits_{t=1}^{m}[b_{tj}]_{+}\omega_{it}$. Therefore

\begin{flalign*}
\sum\limits_{\lambda_{tj}^{\prime}=0}\omega_{tj}^{\prime}b_{th}^{\prime}
&=\sum_{\mbox{\tiny$\begin{array}{c}t:\lambda_{tl}=0 \\\forall l\sim j\end{array}$}}\omega_{tj}^{\prime}b_{th}^{\prime}\\
&=\sum_{\mbox{\tiny$\begin{array}{c}t:\lambda_{tl}=0 \\\forall l\sim j\\t\neq j\end{array}$}}(-\omega_{tj}+\sum\limits_{p=1}^{m}[b_{pj}]_{+}\omega_{tp})(b_{th}+[b_{tj}]_{+}b_{jh}+b_{tj}[-b_{jh}]_{+})\\
&=-\sum_{\mbox{\tiny$\begin{array}{c}t:\lambda_{tl}=0 \\\forall l\sim j\\t\neq j\end{array}$}}\omega_{tj}b_{th}+ \sum_{\mbox{\tiny$\begin{array}{c}t:\lambda_{tl}=0 \\\forall l\sim j\\t\neq j\end{array}$}}\sum\limits_{p=1}^{m}[b_{pj}]_{+}\omega_{tp}b_{th}\\
&-\sum_{\mbox{\tiny$\begin{array}{c}t:\lambda_{tl}=0 \\\forall l\sim j\\t\neq j\end{array}$}}\omega_{tj}[b_{tj}]_{+}b_{jh}+ \sum_{\mbox{\tiny$\begin{array}{c}t:\lambda_{tl}=0 \\\forall l\sim j\\t\neq j\end{array}$}}\sum\limits_{p=1}^{m}[b_{pj}]_{+}\omega_{tp}[b_{tj}]_{+}b_{jh}\\
&-\sum_{\mbox{\tiny$\begin{array}{c}t:\lambda_{tl}=0 \\\forall l\sim j\end{array}$}}\omega_{tj}b_{tj}[-b_{jh}]_{+}+ \sum_{\mbox{\tiny$\begin{array}{c}t:\lambda_{tl}=0 \\\forall l\sim j\\t\neq j\end{array}$}}\sum\limits_{p=1}^{m}[b_{pj}]_{+}\omega_{tp}b_{tj}[-b_{jh}]_{+}. \;\;\;\; \;\;\;(\bigstar)
\end{flalign*}
Let us deal with these terms one by one. We have:
\begin{equation}\label{t1}
\sum_{\mbox{\tiny$\begin{array}{c}t:\lambda_{tl}=0 \\\forall l\sim j\\t\neq j\end{array}$}}\omega_{tj}b_{th}
=\sum\limits_{\lambda_{tj}=0}\omega_{tj}b_{th}
=0.
\end{equation}

This is because $\omega_{tj}=0$ if $\lambda_{tj}=0$ and there is $l\sim j$ satisfying $\lambda_{tl}\neq 0$.
\begin{flalign}
&\sum_{\mbox{\tiny$\begin{array}{c}t:\lambda_{tl}=0 \\\forall l\sim j\\t\neq j\end{array}$}}\sum\limits_{p=1}^{m}[b_{pj}]_{+}\omega_{tp}b_{th}
\;=\;\sum\limits_{p=1}^{m}\sum_{\mbox{\tiny$\begin{array}{c}t:\lambda_{tl}=0 \\\forall l\sim j\\t\neq j\end{array}$}}[b_{pj}]_{+}\omega_{tp}b_{th}\label{t2}\\
=&\sum\limits_{p=1}^{m}[b_{pj}]_{+}(\sum_{\mbox{\tiny$\begin{array}{c}t:\lambda_{tl}=0 \\\forall l\sim j\end{array}$}}\omega_{tp}b_{th})- \sum\limits_{\lambda_{jp}=0}[b_{pj}]_{+}\omega_{jp}b_{jh}
\;=\;a_{h}[b_{hj}]_{+}+\sum\limits_{\lambda_{tj}=0}[b_{tj}]_{+}\omega_{tj}b_{jh}\nonumber\\
=\;&a_{h}[b_{hj}]_{+}+\sum_{\mbox{\tiny$\begin{array}{c}t:\lambda_{tl}=0 \\\forall l\sim j\\t\neq j\end{array}$}}[b_{tj}]_{+}\omega_{tj}b_{jh}.\nonumber
\end{flalign}

And,
\begin{flalign}
&\sum_{\mbox{\tiny$\begin{array}{c}t:\lambda_{tl}=0 \\\forall l\sim j\\t\neq j\end{array}$}}\sum\limits_{p=1}^{m}[b_{pj}]_{+}\omega_{tp}[b_{tj}]_{+}b_{jh}\label{t3}\\
=&\sum_{\mbox{\tiny$\begin{array}{c}t:\lambda_{tl}=0 \\\forall l\sim j\end{array}$}}\sum_{\mbox{\tiny$\begin{array}{c}p:\lambda_{pl}=0 \\\forall l\sim j\end{array}$}}[b_{pj}]_{+}\omega_{tp}[b_{tj}]_{+}b_{jh}+ \sum_{\mbox{\tiny$\begin{array}{c}t:\lambda_{tl}=0 \\\forall l\sim j\\t\neq j\end{array}$}}\sum_{\mbox{\tiny$\begin{array}{c}p:\exists l^{\prime}\sim j\\\lambda_{pl^{\prime}}\neq 0 \end{array}$}}[b_{pj}]_{+}\omega_{tp}[b_{tj}]_{+}b_{jh}\nonumber\\
=&0.\nonumber
\end{flalign}

The last equality holds because the first term equals its opposite number.
$$\sum_{\mbox{\tiny$\begin{array}{c}t:\lambda_{tl}=0 \\\forall l\sim j\end{array}$}}\sum_{\mbox{\tiny$\begin{array}{c}p:\lambda_{pl}=0 \\\forall l\sim j\end{array}$}}[b_{pj}]_{+}\omega_{tp}[b_{tj}]_{+}b_{jh}
=\sum_{\mbox{\tiny$\begin{array}{c}p:\lambda_{pl}=0 \\\forall l\sim j\end{array}$}}\sum_{\mbox{\tiny$\begin{array}{c}t:\lambda_{tl}=0 \\\forall l\sim j\end{array}$}}[b_{tj}]_{+}\omega_{pt}[b_{pj}]_{+}b_{jh} $$
$$=-\sum_{\mbox{\tiny$\begin{array}{c}t:\lambda_{tl}=0 \\\forall l\sim j\end{array}$}}\sum_{\mbox{\tiny$\begin{array}{c}p:\lambda_{pl}=0 \\\forall l\sim j\end{array}$}}[b_{pj}]_{+}\omega_{tp}[b_{tj}]_{+}b_{jh};$$

and if $b_{pj}\neq 0$, then $p\sim j$ so $\lambda_{tp}=0$, which together with $\lambda_{l^{\prime}p}\neq 0$ induces $\omega_{tp}=0$.

\begin{equation}\label{t4}
\sum_{\mbox{\tiny$\begin{array}{c}t:\lambda_{tl}=0 \\\forall l\sim j\end{array}$}}\omega_{tj}b_{tj}[-b_{jh}]_{+}
=\sum\limits_{\lambda_{tj}=0}\omega_{tj}b_{tj}[-b_{jh}]_{+}
=a_{j}[-b_{jh}]_{+};
\end{equation}

\begin{equation}\label{t5}
\sum_{\mbox{\tiny$\begin{array}{c}t:\lambda_{tl}=0 \\\forall l\sim j\\t\neq j\end{array}$}}\sum\limits_{p=1}^{m}[b_{pj}]_{+}\omega_{tp}b_{tj}[-b_{jh}]_{+}
=\sum\limits_{p=1}^{m}\sum_{\mbox{\tiny$\begin{array}{c}t:\lambda_{tl}=0 \\\forall l\sim j\end{array}$}}[b_{pj}]_{+}\omega_{tp}b_{tj}[-b_{jh}]_{+}
=\sum\limits_{p=1}^{m}[b_{pj}]_{+}[-b_{jh}]_{+}(\sum\limits_{\lambda_{tp}=0}\omega_{tp}b_{tj})
=0.
\end{equation}

Substituting (\ref{t1}),(\ref{t2}),(\ref{t3}),(\ref{t4}) and (\ref{t5}) into ($\bigstar$), we get that
$$\sum\limits_{\lambda_{tj}^{\prime}=0}\omega_{tj}^{\prime}b_{th}^{\prime}=a_{h}[b_{hj}]_{+}-a_{j}[-b_{jh}]_{+},$$
which equals 0 if and only if $a_{h}=cd_{h}$ for all $h\in[1,n]$, where $c\in \mathbb{Z}[q^{\pm\frac{1}{2}}]$.
\vspace{2mm}

Now we will prove that when $\hat\Omega\tilde{B}=c\tilde{D}$, it holds $\hat\Omega^{\prime}\tilde{B}^{\prime}=c\tilde{D}$.

Case 1: $k\neq h$ and there is $l-k$ such that $\lambda_{lh}\neq 0$. Similar to the discussion of Case 1 above,
\begin{equation*}
\begin{array}{rl}
\sum\limits_{\lambda_{th}^{\prime}=0}\omega_{th}^{\prime}b_{th}^{\prime} & =\sum\limits_{\mbox{\tiny$\begin{array}{c}\lambda_{th}=0 \\t\neq k\end{array}$}}\omega_{th}(b_{th}+[b_{tk}]_{+}b_{kh}+b_{tk}[-b_{kh}]_{+}) \\
& =\sum\limits_{\lambda_{th}=0}\omega_{th}b_{th}+\sum\limits_{\lambda_{th}=0}\omega_{th}b_{kh}[b_{tk}]_{+}+\sum\limits_{\lambda_{th}=0}
\omega_{th}b_{tk}[b_{kh}]_{+} \\
& =\sum\limits_{\lambda_{th}=0}\omega_{th}b_{th} \\
& =a_{h}
\end{array}
\end{equation*}

Case 2: $k\neq h$ and for any $l-k,\lambda_{lh}=0$.
\begin{equation*}
\begin{array}{rl}
\sum\limits_{\lambda_{th}^{\prime}=0}\omega_{th}^{\prime}b_{th}^{\prime} & =\omega_{kh}^{\prime}b_{kh}^{\prime}+\sum\limits_{\mbox{\tiny$\begin{array}{c}\lambda_{th}=0 \\t\neq k\end{array}$}} \omega_{th}^{\prime}b_{th}^{\prime} \\
& =(-\omega_{kh}+\sum\limits_{t=1}^{m}[b_{tk}]_{+}\omega_{th})(-b_{kh})+\sum\limits_{\mbox{\tiny$\begin{array}{c}\lambda_{th}=0 \\t\neq k\end{array}$}}\omega_{th}(b_{th}+[b_{tk}]_{+}b_{kh}+b_{tk}[-b_{kh}]_{+}) \\
& =\sum\limits_{\lambda_{th}=0}\omega_{th}b_{th}+[-b_{kh}]_{+}\sum\limits_{\lambda_{th}=0}\omega_{th}b_{tk} \\
& =a_{h}
\end{array}
\end{equation*}

Case 3: $k=h$. We have

\begin{flalign*}
&\sum\limits_{\lambda_{th}^{\prime}=0}\omega_{th}^{\prime}b_{th}^{\prime}\;
=\;\sum_{\mbox{\tiny$\begin{array}{c}t:\lambda_{tl}=0 \\\forall l\sim h\\t\neq h\end{array}$}}(-\omega_{th}+\sum\limits_{p=1}^{m}[b_{ph}]_{+}\omega_{tp})(-b_{th})\;
=\;\sum_{\mbox{\tiny$\begin{array}{c}\lambda_{tj}=0 \\t\neq h\end{array}$}}\omega_{tj}(-b_{th})\\
=&\sum\limits_{\lambda_{th}=0}\omega_{th}b_{th}-\sum_{\mbox{\tiny$\begin{array}{c}t:\lambda_{tl}=0 \\\forall l\sim h\\t\neq h\end{array}$}}\sum\limits_{p=1}^{m}[b_{ph}]_{+}\omega_{tp})b_{th}\;
=\; a_{h}-\sum\limits_{p=1}^{m}[b_{ph}]_{+}(\sum\limits_{\lambda_{tp}=0}\omega_{tp}b_{th})\\
=&a_{h}.
\end{flalign*}

In summary, we obtain $\hat\Omega^{\prime}\tilde{B}^{\prime}=c\tilde{D}$.
\end{proof}

It can be seen from Theorem \ref{exchange} and Lemma \ref{zero} that if we regard a cluster algebra as a quantum cluster algebra with $\Lambda=0$, then the compatible conditions (C2) and (C3) in Theorem \ref{exchange} naturally hold, while (C4) is exactly the compatible condition showed in \cite{GSV} and \cite{N}.

\section{Decomposability of quantum cluster algebras and non-trivial second quantization}

\subsection{Decomposability of quantum cluster algebras}.

First we want to show that the compatibility problem can come down to the cluster indecomposable case in the following sence

\begin{Definition}
(1)\; $\tilde{B}$ is said to be \textbf{decomposable} if there is an non-trivial partition of $[1,m]$, $[1,m]=I_{1}\bigsqcup I_{2}$ so that $b_{ij}=0$ whenever $i\in I_{1},j\in I_{2}$. In this case we denote the composition as $\tilde{B}=\tilde{B}_{I_{1}}\bigsqcup\tilde{B}_{I_{2}}$. Otherwise, $\tilde{B}$ is called \textbf{indecomposable}.

(2)\; A quantum cluster algebra $A_{q}$ is said to be \textbf{cluster indecomposable} if there is an indecomposable matrix $\tilde{B}$ in a seed. Otherwise, $A_q$ is called \textbf{cluster decomposable}.
\end{Definition}
In fact, it is easy to see that once a quantum cluster algebra $A_{q}$ is cluster indecomposable (or cluster decomposable), all $\tilde{B}$ matrices of $A_{q}$ are indecomposable (or decomposable, respectively).

\begin{Proposition}\label{decoposition to pieces}
Let $(\tilde{X},\tilde{B},\Lambda)$ be the initial seed of a quantum cluster algebra $A_{q}$ and $\left\{-,-\right\}$ a Poisson bracket on $A_q$. If $\tilde{B}$ has the decomposition $\tilde{B}=\bigsqcup\limits_{i=1}^{r}\tilde{B}_{I_{i}}$ with indecomposables $\tilde{B}_{I_{i}}$ for $i\in[1,r]$, then $\left\{-,-\right\}$ is compatible with $A_{q}$ if and only if $\left\{-,-\right\}|_{A_{q,I_{i}}}$ is compatible with $A_{q,I_{i}}$ $(i\in[1,r])$ and $\sum\limits_{\lambda_{tj}=0}\omega_{tj}b_{tk}=0$ holds for every $k\in [1,n],j\notin I(k)$, where $A_{q,I_{i}}$ is the quantum cluster subalgebra of $A_q$ generated by $(\tilde{B}_{I_{i}},\Lambda_{I_{i}})$ and $I(k)$ is the one of $I_{i}$ containing $k$.
\end{Proposition}

\begin{proof}
If $\left\{-,-\right\}$ is compatible with $A_{q}$, it follows from the definition of compatibility that $\left\{-,-\right\}|_{A_{q,I_{i}}}$ is compatible with $A_{q,I_{i}}$. And Lemma \ref{exchange} requests that $\sum\limits_{\lambda_{tj}=0}\omega_{tj}b_{tk}=0$ for any $k\in [1,n],j\neq k$. In particular, equation holds when $j\notin I(k)$.\par
On the other hand, we learn from above discussion that $\left\{-,-\right\}$ being compatible with $A_{q}$ is equivalent to the three conditions of Lemma \ref{exchange} holding for every seed. And Lemma \ref{zero} transforms the third condition to $\sum\limits_{\lambda_{tj}=0}\omega_{tj}b_{tk}=cd_{k}\delta_{jk}$ for any $j\in[1,m],k\in[1,n]$. Because of the decomposition of $A_{q}$, we know from the compatibility of $\left\{-,-\right\}|_{A_{q,I_{i}}}$ and $A_{q,I_{i}}$, $i\in[1,r]$ that the first two conditions as well as $\sum\limits_{\lambda_{tj}=0}\omega_{tj}b_{tk}=cd_{k}\delta_{jk},j\in I(k)$ always hold. The remaining are $\sum\limits_{\lambda_{tj}=0}\omega_{tj}b_{tk}=0$ for all $j\notin I(k)$.
\end{proof}

Let $A_{q,I_{1}}, A_{q,I_{2}}$ be two quantum cluster algebras with initial seeds $(\tilde{X}_{I_{1}},\tilde{B}_{I_{1}},\Lambda_{I_{1}})$ and $(\tilde{X}_{I_{2}},\tilde{B}_{I_{2}},\Lambda_{I_{2}})$ respectively, and let $\Theta$ be an $|I_{1}|\times|I_{2}|$ integer matrix satisfying
\begin{equation}\label{condition formula for cluster decomposition}
\left\{\begin{array}{cc}
\tilde{B}_{I_{1}}^{\top}\Theta & =0 \\
\Theta\tilde{B}_{I_{2}} & =0.
\end{array}
\right.
\end{equation}
Define $A_{q,I_{1}}\bigsqcup_{\Theta} A_{q,I_{2}}$ to be the algebra equal to $A_{q,I_{1}}\bigotimes_{\Z[q^{\pm\frac{1}{2}}]} A_{q,I_{2}}$ as a $\Z[q^{\pm\frac{1}{2}}]$-module
with twist multiplication:
\begin{equation}\label{twist}
(a\otimes b)(c\otimes d)=\sum\limits_{i,j}k_{i}l_{j}q^{\frac{1}{2}\bar{r}_{i}^{\top}\Theta \bar{s}_{j}}a\tilde{X}_{I_{1}}^{\bar{s}_{j}}\otimes\tilde{X}_{I_{2}}^{\bar{r}_{i}}d
\end{equation}
for $b=\sum\limits_{i}k_{i}\tilde{X}_{I_{2}}^{\bar{r}_{i}},c=\sum\limits_{j}l_{j}\tilde{X}_{I_{1}}^{\bar{s}_{j}}$, where $\bar{r}_{i},\bar{s}_{j}$ are exponential column vectors.

Let $A_{q,I_{i}}$ be quantum cluster algebras with initial seeds $(\tilde{X}_{I_{i}},\tilde{B}_{I_{i}},\Lambda_{I_{i}})$ for $i\in [1,r]$. Then we can obtain the algebra $\bigsqcup\limits_{i=1}^r A_{q,I_{i}}$ since there is the associativity for multiplication in the sense that
\begin{center}
$(A_{q,I_{1}}\bigsqcup_{\Theta_{1}} A_{q,I_{2}})\bigsqcup_{\Theta^{\prime}} A_{q,I_{3}}=A_{q,I_{1}}\bigsqcup_{\Theta^{\prime\prime}} (A_{q,I_{2}}\bigsqcup_{\Theta_{3}}A_{q,I_{3}})$,
\end{center}
where
\[\Theta^{\prime}=\begin{pmatrix}
\Theta_{2} \\
\Theta_{3}
\end{pmatrix},\quad
\Theta^{\prime\prime}=\begin{pmatrix}
\Theta_{1} & \Theta_{2}
\end{pmatrix}.\]
\begin{Proposition}
Following the above notations, the algebra $A_q=\bigsqcup\limits_{i=1}^r A_{q,I_{i}}$ is a quantum cluster algebra with quantum cluster subalgebras $ A_{q,I_{i}}$ for $i\in [1,r]$.
\end{Proposition}
\begin{proof}
Let $A_{q,I_{1}\bigcup I_{2}}$ be a quantum cluster algebra with initial seed $(\tilde{X}_{I_{1}\bigcup I_{2}},\tilde{B}_{I_{1}\bigcup I_{2}},\Lambda_{I_{1}\bigcup I_{2}})$, where
\[\tilde{B}_{I_{1}\bigcup I_{2}}=\begin{pmatrix}
\tilde{B}_{I_{1}} & O \\
O & \tilde{B}_{I_{2}}
\end{pmatrix},\;\;
\Lambda_{I_{1}\bigcup I_{2}}=\begin{pmatrix}
\Lambda_{I_{1}} & \Theta \\
-\Theta^{\top} & \Lambda_{I_{2}}
\end{pmatrix},\]
$(\tilde{B}_{I_{1}\bigcup I_{2}},\Lambda_{I_{1}\bigcup I_{2}})$ is compatible because of (\ref{condition formula for cluster decomposition}). Then clearly there is an isomorphism
\begin{center}
$\eta:\quad A_{q,I_{1}}\bigsqcup_{\Theta} A_{q,I_{2}}\longrightarrow A_{q,I_{1}\bigcup I_{2}}$
\end{center}
via sending $a\otimes b$ to $ab$. Moreover by induction on $r$, we have $A_{q}=\bigsqcup\limits_{i=1}^r A_{q,I_{i}}\cong A_{q,\bigcup\limits_{i}I_{i}}$ as a quantum cluster algebra.
\end{proof}

From this proposition, we call $A_q=\bigsqcup\limits_{i=1}^r A_{q,I_{i}}$ a {\bf cluster decomposition} of $A_{q,\bigcup\limits_{i}I_{i}}$.

Because of the Proposition \ref{decoposition to pieces}, a Poisson structure $\{-,-\}$ is compatible with $A_{q}$ if and only if $\{-,-\}\mid_{A_{q,I_{i}}}$ is compatible with $A_{q,I_{i}}$ for any $i$ and $\sum\limits_{\lambda_{tj}=0}\omega_{tj}b_{tk}=0$ for any $k\in[1,n],j\notin I(k)$. Since the latter one is easy to be checked, in the sequel we only need to consider a cluster indecomposable quantum cluster algebra $A_{q}$.

\begin{Theorem}\label{cluster decomposition}
Let $A_{q}$ be a quantum cluster algebra with initial seed $(\tilde{X},\tilde{B},\Lambda)$ and $\left\{-,-\right\}$ a compatible Poisson bracket on $A_q$. Assume $\Omega$ is the Poisson matrix of the initial cluster with respect to $\{-,-\}$, $\tilde{B}$ has the decomposition $\tilde{B}=\bigoplus\limits_{i=1}^{r}\tilde{B}_{I_{i}}$ with indecomposables $\tilde{B}_{I_{i}}$ for $i\in[1,r]$, and $A_{q,I_{i}}$ is the quantum cluster indecomposable subalgebra of $A_q$ determined by $(\tilde{B}_{I_{i}},\Lambda_{I_{i}})$. Then,

(i)\; $A_{q}$ has a cluster decomposition $A_{q}\cong\bigsqcup\limits_{i} A_{q,I_{i}}$;

(ii)\;Assume $W$ is the second deformation matrix of the second quantization of $A_q$ at the initial quantum seed induced by $\left\{-,-\right\}$. Then the second quantization of $A_q$ is non-trivial if and only if at least one of the following statements holds:

(1)\; there is at least one $i\in [1,r]$ such that the second quantization of $A_{q,I_{i}}$ is non-trivial;

(2)\; there is $k\in\mathbb{Z}$ and $i,j\in[1,r]$ such that $W|_{I_{i}\times I_{i}}= k\Omega|_{I_{i}\times I_{i}}$ but $W|_{I_{i}\times I_{j}}\neq k\Omega|_{I_{i}\times I_{j}}$, .
\end{Theorem}

\begin{proof}
(i)\; Denote $\Theta_{i}=\Lambda\mid_{(\bigcup\limits_{j\leqslant i}I_{j})\times I_{i+1}}$ for $i\in[1,r-1]$. Because $(\tilde{B},\Lambda)$ is compatible,
\begin{center}
$A_{q,I_{1}}\bigsqcup_{\Theta_{1}}\cdots\bigsqcup_{\Theta_{r-1}}A_{q,I_{r}}$
\end{center}
is well defined. And $\eta$ in previous proof shows an isomorphism between $A_{q}$ and $\bigsqcup\limits_{i} A_{q,I_{i}}$.

(ii)\; Recall that the second quantization induced by $\left\{-,-\right\}$ is trivial if there are matrix decompositions $W=\bigoplus\limits_{i=1}^{s}W_{i}$ and $\Lambda=\bigoplus\limits_{i=1}^{s}\Lambda_{i}$ with $W_{i}$ and $\Lambda_{i}$ indecomposable such that $W_{i}=a_{i}\Lambda_{i}$ for some $a_{i}\in\mathbb{Z}$. Hence the second quantization is non-trivial if above decomposition of $W$ and $\Lambda$ have different size or there is $i\in[1,s]$ such that $W_{i}\neq k\Lambda_{i}$ for any $k\in \mathbb Z$. The only two possible cases are that it induces non-trivial second quantization on $A_{q,I_{i}}$ for some $i$ or otherwise there is $i,j\in[1,r]$ such that $W|_{I_{i}\times I_{j}}\neq k\Omega|_{I_{i}\times I_{j}}$, where $k\in\mathbb{Z}$ satisfying $W|_{I_{i}\times I_{i}}= k\Omega|_{I_{i}\times I_{i}}$.
\end{proof}

\begin{Remark}
(1) When $A_{q}$ degenerates to a cluster algebra $A$ as $q\rightarrow 1$, $(a\otimes b)(c\otimes d)=ac\otimes bd$.

(2) When $A_q$ is a quantum cluster algebra without coefficients, $B$ is invertible. Hence by (\ref{condition formula for cluster decomposition}), $\Theta=O$. So we also have $(a\otimes b)(c\otimes d)=ac\otimes bd$.

Then in these cases (1) and (2), $\bigsqcup$ is exactly the tensor product $\otimes$ and the above cluster decomposition coincides with $A=\bigotimes\limits_{i}A_{I_{i}}$ or $A_q=\bigotimes\limits_{i}A_{q, I_{i}}$.
\end{Remark}

\subsection{Quantum cluster algebras with non-trivial second quantization}

\subsubsection{The non-trivial example from a quantum algebra}.

Firstly, we give a simple example of a quantum cluster algebra which has a non-trivial second quantization.

The quantum coordinate algebra (or say, quantum matrix algebra) $Fun_{\C}(SL_{q}(2))$ (see \cite{M,Li}) is generated by $a, b, c, d$ with relations:
\begin{equation}\label{sl2a}
ab=q^{-1}ba,~ac=q^{-1}ca,~db=qbd,~dc=qcd,~bc=cb,ad-da=(q^{-1}-q)bc
\end{equation}
and
\begin{equation}\label{sl2b}
ad-q^{-1}bc=1,
\end{equation}
where $0\ne q\in\mathbb{C}$ is a parameter.

In fact, $Fun_{\C}(SL_{q}(2))$ has a quantum cluster structure, see \cite{LH}. Below is our explanation.

Let $\mathbb{P}=\mathbb{C}[b,c]$. Here $b$ commutes with $c$ by (\ref{sl2a}).
In the $1$-regular tree $T_1$: $\xymatrix@C=0.75cm{t_0\bullet \ar@{-}[r] & \bullet t_1}$, we assign the quantum seed $\Sigma(t_0)=(\widetilde{X}(t_0),\widetilde{B(t_0)},\Lambda(t_0))$ on the vertex $t_{0}$, where $X(t_0)=\{a\}$, $X_{fr}=\{b,c\}$, $X_{t_0}^{e_1}=a$, $X_{t_0}^{e_2}=X^{e_2}=b$, $X_{t_0}^{e_3}=X^{e_3}=c$; $\Lambda(t_0)=\left(\begin{array}{ccc}0&-1&-1\\1&0&0\\1&0&0\end{array}\right),~\widetilde{B}(t_0)=\left(\begin{array}{c}0\\1\\1\end{array}\right)$. It can be verified that $\Lambda(t_0)^\top\widetilde{B}(t_0)=\left(\begin{array}{c}2\\0\\0\end{array}\right)$, i.e, $(\widetilde{B}(t_0), \Lambda(t_0))$ is a compatible pair.

Moreover, let $X(t_1)=\{d\}$, $X_{t_1}^{e_1}=d$. Then, according to (\ref{sl2b}), we have
\[d=q^{-1}a^{-1}bc+a^{-1}=X^{-e_1+[b_1(t_0)]_+}_{t_0}+X^{-e_1+[-b_1(t_0)]_+}_{t_0},\]
which exactly means the mutation of cluster variables, i.e. $\mu_1(X_{t_0}^{e_1}))=X_{t_1}^{e_1}$. And by mutation of matrices, we have $\Lambda(t_1)=\mu_1(\Lambda(t_0))=\left(\begin{array}{ccc}0&1&1\\-1&0&0\\-1&0&0\end{array}\right)$, $\widetilde{B}(t_1)=\mu_{1}(\widetilde{B}(t_0))=\left(\begin{array}{c}0\\-1\\-1\end{array}\right)$.

Following (\ref{sl2a}), the relations of quantum tori:
$$X_{*}^{e_i}X_{*}^{e_j}=q^{\frac{1}{2}\Lambda_{*}(e_i,e_j)}X^{e_i+e_j},\;\;\;\;\forall i,j=1,2,3, *=t_0, t_1$$
hold.

Therefore, through the above discussion, $Fun_{\C}(SL_{q}(2))$ can be realized as the $\mathbb{Z}[q^{\pm\frac{1}{2}}]\mathbb P$-quantum cluster algebra, i.e. $Fun_{\C}(SL_{q}(2))=A_q(\Sigma(t_0))$.

Now we can give the second quantization of $Fun_{\C}(SL_{q}(2))$ according to its compatible Poisson structures. By (C1*), we have $\tilde{B}(t_0)^{\top}W(t_0)=c(D\;O)$ for some $c\in\mathds{Z}[q^{\pm\frac{1}{2}}] $, then it can be induced that the second deformation matrix $W(t_{0})$ must be of the form
\begin{equation}\label{wmatrix}
W(t_{0})=\begin{pmatrix}
0 & -w_{1} & -w_{2} \\
w_{1} & 0 & 0 \\
w_{2} & 0 & 0
\end{pmatrix}
\end{equation}
where $w_{1}+w_{2}\neq 0$. Moreover, it can be checked easily that the matrix $W(t_{0})$ in (\ref{wmatrix}) satisfies (C2), (C3) and (C4). So, such matrix $W(t_{0})$ is what we need. Next, by mutation formula, we have $W(t_{1})=\begin{pmatrix}
0 & w_{1} & w_{2} \\
-w_{1} & 0 & 0 \\
-w_{2} & 0 & 0
\end{pmatrix}$.

According to definition, the second quantization induced by $(\widetilde{B}(t_0),\Lambda(t_0),W(t_0))$ is trivial if and only if $W(t_0)=h\Lambda(t_0)$ for some constant $h$, which means $w_{1}=w_{2}$. Therefore, when $w_{1}+w_{2}\neq 0$ and $w_{1}\neq w_{2}$, the obtained second quantization $A_{p,q}$ of $Fun_{\C}(SL_{q}(2))$ is non-trivial.

In this case, the relations of quantum tori are
$$X_{*}^{e_i}X_{*}^{e_j}=p^{\frac{1}{2}W_{*}(e_i,e_j)}q^{\frac{1}{2}\Lambda_{*}(e_i,e_j)}X^{e_i+e_j},\;\;\;\;\forall i,j=1,2,3, *=t_0, t_1.$$

Following these relations and $X_{t_0}^{e_1}=a,\; X_{t_0}^{e_2}=b,\; X_{t_0}^{e_3}=c,\; X_{t_1}^{e_1}=d$, the secondly quantized cluster algebra $A_{p,q}$ of $Fun_{\C}(SL_{q}(2))$ can be realized as the $\mathds{Z}[q^{\pm\frac{1}{2}}]$-algebra generated by $a,b,c,d$ satisfying the relations as follows:
\begin{equation*}
ab=r^{-1}ba,~ac=s^{-1}ca,~db=rbd,~dc=scd,~bc=cb,ad-da=[(rs)^{-\frac{1}{2}}-(rs)^{\frac{1}{2}}]bc
\end{equation*}
and
\begin{equation*}
ad-(rs)^{-\frac{1}{2}}bc=1,
\end{equation*}
where $r=p^{w_{1}}q,s=p^{w_{2}}q$. Now, also write $A_{p,q}$ as $A_{p,q}(SL(2))$.

\begin{Remark}
We know from \cite{JL} that the 2-parameters quantum coordinate algebra $Fun_{\mathds{C}}(GL_{r,s}(2))$ is generated by $t_{ij}$, $det_{r,s}^{\pm 1}$ with relations:
\[t_{11}t_{12}=r^{-1}t_{12}t_{11},\quad t_{11}t_{21}=st_{21}t_{11},\quad t_{21}t_{22}=r^{-1}t_{22}t_{21},\quad t_{12}t_{22}=st_{22}t_{12},\]
\[t_{12}t_{21}=rst_{21}t_{12},\quad t_{11}t_{22}-t_{22}t_{11}=(s-r)t_{21}t_{12},\quad det_{r,s}det_{r,s}^{-1}=det_{r,s}^{-1}det_{r,s}=1,\]
\[det_{r,s}t_{ij}=(rs)^{i-j}t_{ij}(det_{r,s}),\quad det_{r,s}=t_{11}t_{22}-st_{21}t_{12}=t_{22}t_{11}-rt_{21}t_{12}=t_{11}t_{22}-r^{-1}t_{12}t_{21}.\]

If we consider this $2$-parameters quantum algebra from $GL_{r,s}(2)$ to $SL_{r,s}(2)$, then we have $det_{r,s}=1$.
Replacing it into the relation $det_{r,s}t_{ij}=(rs)^{i-j}t_{ij}(det_{r,s})$, we get $r=s^{-1}$, that is, 2-parameter quantum algebra $Fun(GL_{r,s}(2))$ is degenerated into one parameter quantum algebra $Fun(GL_r(2))$. It means that this method of 2-parameters quantization $Fun(GL_{r,s}(2))$ of $Fun(GL_r(2))$ has no effect on the special quantum linear group $SL_r(2)$.

On the other hand, the second quantization $A_{p,q}(SL(2))$ in the above example is for the special quantum linear group $SL_q(2)$. So, we can say that the second quantization $A_{p,q}(SL(2))$ provides a way to realize two-parameters quantization of the special quantum linear group $SL_q(2)$, as a parallel supplement to the method of two parameters quantization of the general quantum group.
\end{Remark}

\subsubsection{Non-trivial secondly quantized cluster algebras via cluster extensions}.

In this part, we will present a class of quantum cluster algebras with non-trivial second quantization via cluster extensions.

\begin{Lemma}\label{exampleone}

For $m>l\geqslant n$ and $a\in\Z$, let
\begin{equation*}
\tilde{B}=
\begin{pmatrix}
B \\
L \\
\end{pmatrix}_{m\times n}\quad
\Lambda=
\begin{pmatrix}
\Lambda_{0} &O \\
O & O
\end{pmatrix}_{m\times m}\quad
W=
\begin{pmatrix}
a\Lambda_{0} & O \\
O & P
\end{pmatrix}_{m\times m}
\end{equation*}
In these matrices, the sizes of blocks are correspondent, $B$ is an $l\times n$ symmetrizable integer matrix with skew-symmetrizer $D$, $L$ is an integer matrix and $P$ is a non-zero skew-symmetric integer matrix. Assume $\tilde{B}^{\top}\Lambda=(D\text{ }O)$, $L^{\top}P=O$ and under any sequence of mutations, the sub-matrix $L$ of $\tilde{B}$ always maintains column sign coherent.

Then $(\tilde{B},\Lambda,W)$ is compatible and determines a quantum cluster algebra $A_{q}$. The corresponding second quantization $A_{p,q}$ of $A_q$ is non-trivial.
\end{Lemma}

\begin{proof}
$\tilde{B}^{\top}W=a(B^{\top}\Lambda_{0}\quad L^{\top}P)=a(D\text{ }O)$. Because of the column sign coherence of $L$, $[L]_{+}^{\top}P=O$, which ensures the zero blocks of $W$ remain after any sequence of mutations. Hence it can be verified that for any sequences of mutations $\mu_{I}$, there is always
\[\mu_{I}(\Lambda)=\begin{pmatrix}
\mu_{I}(\Lambda_{0}) & O \\
O & O
\end{pmatrix},\quad
\mu_{I}(W)=\begin{pmatrix}
\mu_{I}(a\Lambda_{0}) & O \\
O & P
\end{pmatrix}.\]
Then according to the definition of the compatible triple, we see easily that $(\tilde{B},\Lambda,W)$ is compatible. And by definition the second quantization is non-trivial.
\end{proof}

As a concrete example if we let
\begin{equation*}
\tilde{B}=
\begin{pmatrix}
0 & 1 \\
-1 & 0 \\
1 & 0 \\
1 & 0 \\
1 & 0
\end{pmatrix},\qquad
\Lambda=
\begin{pmatrix}
0 & 1 & 0 & 0 & 0 \\
-1 & 0 & 0 & 0 & 0 \\
0 & 0 & 0 & 0 & 0 \\
0 & 0 & 0 & 0 & 0 \\
0 & 0 & 0 & 0 & 0
\end{pmatrix},\qquad
W=
\begin{pmatrix}
0 & a & 0 & 0 & 0 \\
-a & 0 & 0 & 0 & 0 \\
0 & 0 & 0 &-b & b \\
0 & 0 & b & 0 & -b \\
0 & 0 & -b & b & 0
\end{pmatrix},
\end{equation*}
then $(\tilde{B},\Lambda,W)$ is compatible for any $a,b\in\mathbb{Z}$ and it raises a non-trivial second quantization $A_{p,q}$.

For any quantum seed $\Sigma=(\tilde{X},\tilde{B},\Lambda)$ of a quantum cluster algebra $A_q$, we call a quantum seed $\Sigma'=(\tilde{X}^{\prime},\tilde{B}^{\prime},\Lambda^{\prime})$ a {\bf cluster extension} of $\Sigma$ if $m'>m$ and $\tilde{X}^{\prime}=\tilde{X}\bigcup\{X_{m+1},\cdots,X_{m^{\prime}}\}$ with $X_{m+1},\cdots,X_{m^{\prime}}$ as extra frozen variables, in which the first $m$ rows of $\tilde{B}^{\prime}$ is $\tilde{B}$ and $\Lambda^{\prime}=\begin{pmatrix}
\Lambda & O \\
O & O
\end{pmatrix}$.
The quantum cluster algebra $A_q$ associated to $\Sigma$ is a specialization of the quantum cluster algebra $A'_q$ associated to $\Sigma^{\prime}$ with $X_{m+1}=\cdots=X_{m^{\prime}}=1$. In this case, we also call the quantum cluster algebra $A'_q$ a {\bf cluster extension} of $A_q$.
\begin{Proposition}
Assume $A_{q}$ is a quantum cluster algebra with the cluster decomposition $A_{q}=\bigsqcup\limits_{i}^rA_{q,I_{i}}$ and $A_{q,I_{i}}^{\prime}$ is a cluster extension of $A_{q,I_{i}}$ respectively for $i=1,\cdots,r$. Then $\bigsqcup\limits_{i}^rA_{q,I_{i}}^{\prime}$ is a cluster extension of $A_{q}$.
\end{Proposition}
\begin{proof}
This result can be verified by comparing initial $B$-matrices.
\end{proof}

So we can focus on the indecomposable case as follows.

\begin{Theorem}\label{ext}
Let $\Sigma=(\tilde{X},\tilde{B},\Lambda)$ be an arbitrary quantum seed of a quantum cluster indecomposable algebra $A_q$. Then there is a cluster extension $\Sigma^{\prime}=(\tilde{X}^{\prime},\tilde{B}^{\prime},\Lambda^{\prime})$ of $\Sigma$ such that the cluster extension $A'_q$ of $A_q$ admits a non-trivial second quantization $A'_{p,q}$.\footnote{We are grateful to Zhuoheng He for the discussion about the above matrix equation. }
\end{Theorem}
\begin{proof}
Choose a fixed seed $\Sigma_{t_{0}}=(\tilde{X}_{t_{0}},\tilde{B}_{t_{0}},\Lambda_{0})$ in $A_{q}$. Let $\A$ be the cluster algebra with principal coefficients having $\begin{pmatrix}
B_{t_{0}} \\
I_{n}
\end{pmatrix}$ as initial exchange matrix. Then there is a seed in $\A$ having $\begin{pmatrix}
B \\
C
\end{pmatrix}$ as exchange matrix with principal part $B$. For any choice of the initial seed $\Sigma_{t_{0}}$ of $A_q$, there may be more than one $C$-matrices $C$ such that $\begin{pmatrix}
B \\
C
\end{pmatrix}$ is an extended exchange matrix of a seed of $\A$. We just choose arbitrary one.

Let $\Sigma'=(\tilde{X}^{\prime},\tilde{B}^{\prime},\Lambda^{\prime})$ be a cluster extension of $\Sigma$ such that $\tilde{B}^{\prime}=\begin{pmatrix} \tilde{B} \\
C\\
C^{\prime} \end{pmatrix}$,
where $C^{\prime}$ consists of some $s(>n+1)$ (maybe, repeated) rows from $C$.

Since $C^{\prime}$ consists of several rows from $C$, according to the definition of mutations of exchange matrices $\tilde B$, $\mu_k\cdots \mu_1(C')$ still consists of the corresponding rows of $\mu_k\cdots \mu_1(C)$ after any sequence of mutations $\mu_1, \cdots, \mu_k$. Therefore, $C^{\prime}$ inherits column sign coherence from $C$, and thus, $\begin{pmatrix}
C \\
C^{\prime}
\end{pmatrix}$
maintains column sign coherence after any sequence of mutations.

In order to apply Lemma \ref{exampleone}, we need that there is a non-zero skew-symmetric integer $(n+s)\times (n+s)$-matrix $P=(p_{ij})$ satisfying that
\begin{equation}\label{P}
\begin{pmatrix}
C \\
C^{\prime}
\end{pmatrix}^{\top}P=O_{n\times (n+s)}.
\end{equation}

Consider this matrix equation (\ref{P}) for $P$. We claim that it has a non-zero skew-symmetric matrix solution $P$.

Let $\bar p_{i}$ be the $i$-th column of $P$, that is, $P=(\bar p_{1} \; \cdots \; \bar p_{n+s})$. Then (\ref{P}) can be written as
\[\begin{pmatrix}
C \\
C^{\prime}
\end{pmatrix}^{\top}(\bar p_{1} \; \cdots \; \bar p_{n+s})=O,\]
which is equivalent to the system of linear equations
\begin{equation}\label{A}
\begin{pmatrix}
\begin{pmatrix}
C^{\top}&(C^{\prime})^{\top}
\end{pmatrix} & O & O \\
O & \ddots & O \\
O & O & \begin{pmatrix}
C^{\top}&(C^{\prime})^{\top}
\end{pmatrix}
\end{pmatrix}_{n(n+s)\times (n+s)^2}
\begin{pmatrix}
\bar p_{1} \\
\vdots \\
\bar p_{n+s}
\end{pmatrix}_{(n+s)^2\times 1}=O_{n(n+s)\times 1}
\end{equation}
This system contains $n(n+s)$ linear equations.

Additionally, in order to make $P$ a skew-symmetric matrix, we need also the following system of linear equations:
\begin{equation}\label{B}
\left\{\begin{array}{ccl}
p_{ii} & = 0,&\;\; \text{if }\; i=1,\dots, n+s; \\
p_{ij}+p_{ji} & = 0,&\;\; \text{if }\; i\not=j,\; \text{for}\; i,j=1,\dots, n+s.
\end{array}
\right.
\end{equation}
This system contains $(n+s)+\frac{1}{2}(n+s)(n+s-1)$ linear equations.

Combining (\ref{A}) with (\ref{B}), we obtain a system of $r$ linear equations, denoted as $(I)$, whose solution $\{p_{ij}:\; i,j=1,\dots,n+s\}$ always forms a skew-symmetric matrix $P=(p_{ij})$, where
$$r=n(n+s)+(n+s)+\frac{1}{2}(n+s)(n+s-1)=(n+s)(\frac{3}{2}n+\frac{1}{2}s+\frac{1}{2}).$$

It is easy to see that $(n+s)^2-r>0$ when $s>n+1$. Then, in the system $(I)$ of linear equations, the number of undetermined elements is always larger than the number of linear equations in this case. Moreover, note that the system $(I)$ is homogeneous. Hence, the system $(I)$ has non-zero solution, say $\{p^0_{ij}:\; i,j=1,\dots,n+s\}$. Thus, as a solution of the matrix equation (\ref{P}), we can obtain a non-zero skew-symmetric (integer) matrix $P_0=(p^0_{ij})$.

Then, by Lemma \ref{exampleone}, $\tilde{B}^{\prime},\; \Lambda^{\prime}$ and $W=
\begin{pmatrix}
\Lambda & O \\
O & P_{0}
\end{pmatrix}_{(m+n+s)\times (m+n+s)}$ are compatible and moreover, they induce a non-trivial second quantization of the quantum cluster algebra $A'_q$ associated to $\Sigma^{\prime}$.
\end{proof}

We call a quantum cluster algebra $A_q$ to be {\bf with almost principal coefficients} if its initial extended exchange matrix has the form
\[\tilde{B}_{t_0}=\begin{pmatrix}
B_{t_0} \\
I_{n} \\
J
\end{pmatrix},\]
where $J$ consists of some (maybe, repeated) rows from $I_{n}$.

Particularly, in the above theorem and its proof, when $\Sigma_{t_0}=(X_{t_0},B_{t_0},\Lambda_{t_0})$ is a quantum seed of a quantum cluster algebra $A^{o}_q$ without coefficients, let $\Sigma^{\prime}_{t_0}=(\tilde{X}^{\prime}_{t_0},\tilde{B}^{\prime}_{t_0},\Lambda^{\prime}_{t_0})$ be a cluster extension of $\Sigma_{t_0}$ such that $\tilde{B}^{\prime}_{t_0}=\begin{pmatrix} B_{t_0} \\
I_n \\
J
\end{pmatrix}_{m\times n}$,
where $J_{s\times n}$ ($s>n+1$) consists of some (maybe, repeated) rows from $I_n$. Then, the quantum cluster algebra $A_q'$ associated to $\Sigma^{\prime}_{t_0}$ is a quantum cluster algebra with almost principal coefficients.

As done in the proof of Theorem \ref{ext}, let $\Sigma^{\prime}=(\tilde{X}^{\prime},\tilde{B}^{\prime},\Lambda^{\prime})$ be any seed in $A_q'$ mutation equivalent to $\Sigma^{\prime}_{t_0}$ with $\tilde{B}^{\prime}=\begin{pmatrix} B \\
C \\
C^{\prime}
\end{pmatrix}$,
where $C$ is a $C$-matrix and $C^{\prime}_{s\times n}$ consists of some (maybe, repeated) rows from $C$.

Then by Theorem \ref{ext}, $A_{q}'$ admits non-trivial second quantization since $C^{\prime}$ has more than $n+1$ rows. In this situation, $m=n+n+s$. So, $s>n+1$ if and only if $m>3n+1$. Therefore we have the following corollary:

\begin{Corollary}\label{almost}
Any quantum cluster algebra $A'_q$ with almost principal coefficients having $m\times n$ extended exchange matrices always admits a non-trivial second quantization $A'_{p,q}$ when $m>3n+1$.
\end{Corollary}

\vspace{4mm}

\section{Second quantization of quantum cluster algebras without coefficients}
In the following assume $A_{q}$ is a quantum cluster algebra without coefficients and $(X,B,\Lambda)$ is the initial seed. In this case, $B$ and $\Lambda$ are both invertible. Suppose $B$ has decomposition $B=B_{1}\bigsqcup B_{2}\bigsqcup\cdots\bigsqcup B_{s}$, where $B_{r}$ is indecomposable. Let the set of indices of $B_{r}$ as submatrix be $I_{r}$ and $X(r)=\left\{X_{j}\mid j\in I_{r}\right\}$. As we said before, the coefficients-free condition indicates the invertibility of $B$, and thus the invertibility of $B_{r}$ for each $r$. Meanwhile, $B^{\top}\Lambda=D$. Hence $\Lambda$ has decomposition $\Lambda=\Lambda_{1}\bigsqcup\Lambda_{2}\bigsqcup\cdots\bigsqcup\Lambda_{s}$, where $\Lambda_{r}$ is indecomposable and the set of indices of $\Lambda_{r}$ is exactly $I_{r}$. Then $(X(r),B_{r},\Lambda_{r})$ is a seed from which we get a quantum cluster subalgebra $A_{q,I_{r}}$. Moreover, it follows from the decompositions of $B$ and $\Lambda$ the decomposition of quantum cluster algebras $A_{q}=A_{q,I_{1}}\bigotimes A_{q,I_{2}}\bigotimes\cdots\bigotimes A_{q,I_{s}}$.

\begin{Definition}
Let $\left\{-,-\right\}$ be a Poisson bracket on a quantum cluster algebra $A_{q}$. For a seed $\Sigma=(\tilde{X},\tilde B,\Lambda)$, let $B=B_{1}\bigsqcup\ldots\bigsqcup B_{s}$ with $B_{r}$ indecomposable. $\left\{-,-\right\}$ is a \textbf{locally standard Poisson structure} if $\{X_{i},X_{j}\}=0$ when $i$ and $j$ are from different $I_{r}$ and $\left\{-,-\right\}$ is of standard poisson structure on each $X(r)$, i.e, $\{X_{i},X_{j}\}=a_{r}[X_{i},X_{j}]$, where $i,j\in I_{r},a_{r}\in \mathbb{Z}[q^{\pm\frac{1}{2}}]$.
\end{Definition}

\begin{Theorem}\label{r2}
Let $A_{q}$ be a quantum cluster algebra without coefficients. Then a Poisson structure $\left\{-,-\right\}$ on $A_{q}$ is compatible with $A_{q}$ if and only if it is locally standard on $A_{q}$.
\end{Theorem}
\begin{proof}
``If":
The compatibility of a Poisson structure on $A_q$ can be verified directly using the definition of locally standard Poisson brackets.

``Only if": First let $A_{q}$ be indecomposable. As Proposition \ref{BW=D} says, $B^{\top}W=cD$, where $c\in\mathbb{Z}[q^{\pm\frac{1}{2}}]$ and $D$ is a skew-symmetrizer of $B$. So $W=c(B^{\top})^{-1}D=c\cdot a\Lambda$, where $a\lambda_{ij}\in\mathbb{Z}[q^{\pm\frac{1}{2}}]$ for any $i,j\in[1,m]$. According to the definition, $\omega_{ij}=c[\lambda_{ij}]_{q^{\frac{1}{2}}}$, i.e, \{$X_{i},X_{j}$\}=c[$X_{i},X_{j}$]. Therefore for any $\alpha_i,\beta_j\in \mathbb{Z}^{n}$, \{$\prod \limits_{i=1}^{n}X_{i}^{\alpha_{i}},\prod \limits_{j=1}^{n}X_{j}^{\beta_{j}}$\}=a[$\prod \limits_{i=1}^{n}X_{i}^{\alpha_{i}},\prod \limits_{j=1}^{n}X_{j}^{\beta_{j}}$], which means that $\left\{-,-\right\}$ is standard.

Then according to Proposition \ref{decoposition to pieces}, $\left\{-,-\right\}$ is locally standard for any quantum cluster algebra without coefficients.
\end{proof}
%
%
%

As showed in the last section, any compatible Poisson structure on $A_{q}$ are locally standard. In particular, it is standard when restricted on each quantum cluster subalgebra $A_{q}(i)$. Therefore without loss of generality, in the following we can assume $A_{q}$ is indecomposable with a standard Poisson structure. Then in a compatible triple, using the proof of Theorem \ref{r2}, we have:
\begin{equation*}
\Omega=a\begin{pmatrix}
0 & [\lambda_{12}]_{q^{\frac{1}{2}}} & \cdots & [\lambda_{1n}]_{q^{\frac{1}{2}}} \\
[\lambda_{21}]_{q^{\frac{1}{2}}} & 0 & \cdots & [\lambda_{2n}]_{q^{\frac{1}{2}}} \\
\vdots & \vdots & \ddots & \vdots \\
[\lambda_{n1}]_{q^{\frac{1}{2}}} & [\lambda_{n2}]_{q^{\frac{1}{2}}} & \cdots & 0
\end{pmatrix}.
\end{equation*}
where $a$ is an integer. And $W_{ij}=a\lambda_{ij}$ for any $i,j\in[1,n]$. Therefore in this case, from (\ref{second quantum exchange}) we obtain that
\begin{equation*}
Y_{t}^{e_{i}}Y_{t}^{e_{j}}=(p^{a}q)^{\frac{1}{2}\lambda_{ij}}Y_{t}^{e_{i}+e_{j}},\forall i,j\in[1,n].
\end{equation*}
So, the secondly quantized cluster algebra $A_{p,q}$ is essentially a quantum cluster algebra with one parameter.

In this aspect, it is showed above that:
\begin{Corollary}\label{trivial}
The second quantization of a quantum cluster algebra without coefficients is always trivial.
\end{Corollary}
\vspace{4mm}

{\bf Acknowledgements:}\;

{\em This project is supported by the National Natural Science Foundation of China(No.11671350) and the Zhejiang Provincial Natural Science Foundation of China (No. LY19A010023).}

{\em Fang Li thanks to Naihuan Jing and Zhaobing Fan for their encouragement and discussion at the beginning of this work.}

\end{document}